\documentclass{article}

\usepackage{amsmath}
\usepackage{amsthm}
\usepackage{amssymb}
\usepackage{mathrsfs}
\usepackage{graphicx}
\usepackage{verbatim}
\usepackage{tikz}
\usepackage[all]{xy}
\usepackage{graphics}
\usepackage[colorlinks=true]{hyperref}
\hypersetup{urlcolor=blue, citecolor=red}

\newtheorem{theorem}{Theorem}[section]
\newtheorem{proposition}[theorem]{Proposition}
\newtheorem{lemma}[theorem]{Lemma}
\newtheorem{con lemma}[theorem]{Continuity Lemma}
\newtheorem{corollary}[theorem]{Corollary}

\theoremstyle{definition}%da qui in poi negli environment introdotti il testo non � in italico 

\newtheorem{definition}{Definition}

\theoremstyle{remark}
\newtheorem{remark}{Remark}

\newcommand{\R}{\mathbb{R}}
\newcommand{\Z}{\mathbb{Z}}
\newcommand{\N}{\mathbb{N}}

\def\beq{\begin{equation}}
\def\eeq{\end{equation}}
\def\pa{\partial}
\def\f{\varphi}
\def\eps{\varepsilon}
\def\wc{\rightharpoonup}
\def\o{\omega}
\def\wt{\widetilde}
\def\wh{\widehat}
\def\a{\alpha}
\def\n{\nabla}
\def\t{\vartheta}
\def\d{\delta}
\def\g{\gamma}
\def\s{\sigma}
\def\l{\lambda}
\def\r{\rho}

\title{\sc Symbolic dynamics: from the $N$-centre to the $(N+1)$-body problem, a preliminary study}
\author{Nicola Soave}

\begin{document}

%%%%%%%%%%%%%%%%%%%%%%%%%%%%%%%%%%%%%%%%%%%%%%%%%%%%%%%%%%%%%%%%%%%%%%%%%%%%%%%
\maketitle
%%%%%%%%%%%%%%%%%%%%%%%%%%%%%%%%%%%%%%%%%%%%%%%%%%%%%%%%%%%%%%%%%%%%%%%%%%%%%%%

{\footnotesize
 \centerline{Dipartimento di Ma\-t\-ema\-ti\-ca e Applicazioni, Universit\`a degli Studi di Milano-Bicocca}
   \centerline{Via Cozzi 53}
   \centerline{20125 Milano, Italy.} 
   \centerline{LAMFA, CNRS UMR 7352, Université de Picardie Jules Verne} 
\centerline{33 rue Saint-Leu}
\centerline{80039 Amiens, France.}}
\centerline{email: n.soave@campus.unimib.it}

\begin{abstract}
\noindent We consider a rotating $N$-centre problem, with $N \geq 3$ and homogeneous potentials of degree $-\a<0$, $\a \in [1,2)$. We prove the existence of infinitely many collision-free periodic solutions with negative and small \emph{Jacobi constant} and small values of the angular velocity, for any initial configuration of the centres. We will introduce a Maupertuis' type variational principle in order to apply the \emph{broken geodesics technique} developed in \cite{SoTe}. Major difficulties arise from the fact that, contrary to the classical Jacobi length, the related functional does not come from a Riemaniann structure but from a Finslerian one. Our existence result allows us to characterize the associated dynamical system  with a symbolic dynamics, where the symbols are given partitions of the centres in two non-empty sets.
\end{abstract}


\section{Introduction and main results}\label{intro}
In the classical $N$-centre problem it is investigated the motion of a test particle of null mass under the gravitational force fields of $N$ fixed heavy bodies (the \emph{centres}): if $c_k$ and $m_k$ denote respectively the position and the mass of the $k$-th centre, the motion equation is 
\beq\label{N centri}
\ddot{x}(t)=  -\sum_{k=1}^N \frac{m_k}{|x(t)-c_k|^3} (x(t)-c_k)= \n \left.\left(\sum_{k=1}^N\frac{m_k}{|x-c_k|}\right)\right|_{x=x(t)},
\eeq
where $x=x(t) \in \R^2$ denotes the position of the particle at time $t \in \R$; basic references for such a problem are \cite{Bo, BoNe, Di, KlKn, Kn, KnTa, SoTe} and the references therein. In this paper we consider $\a$-gravitational potentials of type
\[
V(x)=  \sum_{k=1}^N\frac{m_k}{\a|x-c_k|^\a} \qquad \a \in [1,2).
\]
Of course, for $\a=1$ we get the classical Newtonian potential; moreover, we assume that the centres are not fixed, but rotate according to the law $\xi_k(t):= \exp{\{i \nu t\}} c_k$. Under this assumption, the equation for the motion of the test particle is  
\beq\label{motion eq}
\ddot{x}(t)=-\sum_{k=1}^N \frac{m_k}{|x(t)-e^{i \nu t}c_k|^{\a+2}}\left(x(t)-e^{i \nu t}c_k\right).
\eeq
We will refer to the research of solutions to this equation as to \emph{the rotating $N$-centre problem} (briefly, the rotating problem). It is convenient to introduce a different frame of reference for $x$, taking into account the rotation of the centres: setting $x(t)=\exp{\{i \nu t\}} z(t)$, equation \eqref{motion eq} becomes
\beq\label{motion eq 2}
\ddot{z}(t)+2\nu i \dot{z}(t)=\nu^2 z(t) - \sum_{k=1}^N \frac{m_k}{|z(t)-c_k|^{\a+2}} \left( z(t)- c_k \right).
\eeq
We introduce $\Phi_\nu(z):= \nu^2|z|^2 /2+V(z)$, so that \eqref{motion eq 2} can be written as
\[
\ddot{z}(t)+2 \nu i \dot{z}(t)= \nabla \Phi_\nu(z(t)).
\]
Since the terms in $z$ and $\dot{z}$ are multiplied by powers of $\nu$, the idea is that if $|\nu|$ is sufficiently small, then equation \eqref{motion eq 2} can be regarded as a perturbation of the planar $N$-centre problem, which we dealt with in \cite{SoTe}. Note that, contrary to \eqref{N centri}, equation \eqref{motion eq 2} is not a conservative system; however, it is possible to find a first integral defining 
\[
J_{\nu}(z,\dot{z}):= \frac{1}{2}|\dot{z}|^2-\Phi_\nu(z).
\]
The value $h=J_\nu(z(t),\dot{z}(t))$, which is the same for every $t \in I$, is called the \emph{Jacobi constant}, in analogy with the same integral of the circular restricted $(N+1)$-body problem (see the discussion below for the relationship between the rotating problem and the restricted one). Note the similarity between $J_\nu$ and the usual energy function $H(z,\dot{z})=|\dot{z}|^2/2 - V(z)$: it results $H=J_0$.\\
In this paper we generalize the approach already developed in \cite{SoTe}, proving the existence of infinitely many collision-free periodic solutions of equation \eqref{motion eq 2} with negative and small (in absolute value) Jacobi constant, provided the angular velocity $|\nu|$ is sufficiently small. As a consequence, for those values of $h$ and $\nu$ we can characterize the dynamical system induced by \eqref{motion eq 2} on the level set
\[
\mathcal{U}_{h,\nu}:=\left\{(z,v) \in \R^4:  J_\nu(z,v)=h \right\}
\]
with a symbolic dynamics, where the symbols are some selected partitions of the centres in two different non-empty sets. Coming back to equation \eqref{motion eq}, this means that for $h<0$ and $|h|, |\nu|$ sufficiently small  we have infinitely many collision-free \emph{relative periodic solutions} (i.e. periodic in the rotating frame of reference) of the rotating problem; this existence result allows to prove the occurrence of symbolic dynamics in a proper submanifold of the phase space (which correspond to $\mathcal{U}_{h,\nu}$ through the transformation $x \leftrightsquigarrow z$).

\paragraph{Motivations.} The $N$-centre problem can be considered as a simplified version of the $(N+1)$-body problem, when one of the bodies is much faster then the others. Therefore, in order to understand if the \emph{broken geodesics technique} we introduced in \cite{SoTe} can be somehow extended to find solutions of the $(N+1)$-body problem, it seems reasonable to start considering an "easy test motion" for the centres, such as the uniformly circular one. This is strictly related to the study of the circular restricted $(N+1)$-body problem, which we briefly recall; assigned $N$ positive masses $m_1,\ldots,m_N$, let us consider any planar central configuration $(c_1,\ldots,c_N)$ of the $N$-body problem. A relative equilibrium of the $N$-body problem is a motion of type $\xi_k(t):=\exp{\{i \nu t \}} c_k$ ($k=1,\ldots,N$), with $\nu \in \R$, i.e. an equilibrium point in a rotating frame of reference with angular velocity $\nu$. The restricted problem consists in studying the motion of a test particle of null mass under the gravitational force field of $N$ bodies (the \emph{primaries}) which move according to a motion of relative equilibrium. This leads to the search of solutions to \eqref{motion eq 2}, but now $\nu$ cannot be considered as a free parameter: indeed, each central configuration determines the unique admissible value of $\nu$ through the relation
\beq\label{config centrali}
\nu^2=\frac{U(\mathbf{c})}{2I(\mathbf{c})}, \quad \text{where} \quad U(\mathbf{c})=\sum_{1\leq j<k\leq N}\frac{m_j m_k}{|c_j-c_k|}, \ I(\mathbf{c})= \frac{1}{2}\sum_{k=1}^N m_k |c_k|^2,
\eeq
see Meyer \cite{Me}. In particular, letting $\nu$ to tend to $0$, the relation \eqref{config centrali} implies that either $m_k \to 0$ for every $k$ or $|c_k| \to 0$ for every $k$; as a consequence, the equation of the restricted problem in the limit case $\nu \to 0$ tends to $\ddot{z}=0$, which has no relation with the $N$-centre problem or the $N$-body problem. As a toy model towards the real restricted $(N+1)$-body problem, we introduce the rotating $N$-centre problem; we point out that the motivation for its study is prevalently mathematical: our goal is to understand if the techniques introduced in \cite{SoTe} are sufficiently robust to survive when we perturb the $N$-centre problem by letting the centres move; the answer is yes, but, as we will see, the extension of our broken geodesics method is not trivial and requires new ideas. Therefore, the generalization to the real restricted problem seems possible, but extremely complicated.

\subsection{Periodic solutions}

Let $\mathcal{P}$ be the set of the possible partitions of the centers in two different non-empty sets. There are exactly $2^{N-1}-1$ such partitions, and to each of them we associate a label:
\[
\mathcal{P}=\left\{P_j:j=1,\ldots,2^{N-1}-1\right\}.
\] 
We give particular labels to those partitions which isolates one centre with respect to the others:
\[
P_j:=\{\{c_j\},\{c_1,\ldots,c_N\}\setminus \{c_j\}\} \qquad j=1,\ldots,N.
\]
The collection of these labels is the subset
\begin{equation}\label{def di P_1}
\mathcal{P}_1:=\{P_j \in \mathcal{P}: j=1,\ldots,N\} \subset \mathcal{P}.
\end{equation}
We define the \emph{right shift} $T_r:  \mathcal{P}^n \to \mathcal{P}^n$ as
\begin{equation*}
T_r((P_{j_1},P_{j_2},\ldots,P_{j_n})) = (P_{j_n},P_{j_1},\ldots, P_{j_{n-1}}),
\end{equation*}
and we say that \emph{$(P_{j_1},\ldots,P_{j_n}) \in \mathcal{P}^n$ is equivalent to $(P_{j_1}',\ldots, P_{j_n}') \in \mathcal{P}^n$} if there exists $m \in \mathbb{N}$ such that
\[
(P_{j_1}',\ldots, P_{j_n}')=T_r^m \left((P_{j_1},\ldots,P_{j_n}) \right).
\]
To describe the first main result which we are going to prove, let us look at Theorem 1.1 of \cite{SoTe}; therein we proved the existence of $\bar{h}<0$ such that for any $h \in (\bar{h},0)$ we can associate to any finite sequence of partition $(P_{j_1},\ldots, P_{j_n}) \in \mathcal{P}^n$ a periodic solution $x_{((P_{j_1},\ldots,P_{j_n}),h)}$ of the $N$-centre problem \eqref{N centri} with energy $h$. Under particular assumptions on $(P_{j_1},\ldots, P_{j_n})$, assumptions which are specified in points ($ii$)-$b$) or ($ii$)-$c$) of the quoted statement, we have to allow collision solutions, but it is always possible (for every $N \geq 3$) to build infinitely many collision-free solutions. We would like to repeat the game associating to a finite sequence of partitions, for sufficiently small values of the absolute value of the Jacobi constant $|h|$ and of the angular velocity $|\nu|$, a periodic solution of equation \eqref{motion eq 2}. In this paper we will put some restrictions on the sequences of partitions which we want to consider; this is motivated by the fact that the rotation of the centres makes impossible the use of some techniques employed in the study of the behaviour of collision-solutions. In this sense we observed in \cite{SoTe} that the study of the collisions requires a distinction among
\[
 1) \ \a=1 \text{ and } N \geq 4, \qquad  2) \ \a=1 \text{ and } N=3, \qquad 3) \ \a \in (1,2).
\]

\noindent We start from the first case.
\begin{theorem}\label{esistenza di soluzioni periodiche}
Let $\a=1$, $N \geq 4$, $c_1, \ldots, c_N \in \mathbb{R}^2$, $m_1,\ldots,m_N \in \R^+$. There exists $\bar{h}_1$ such that, given $h \in (\bar{h}_1,0)$, there is $\bar{\nu}_1= \bar{\nu}_1(h)>0$ such that to each $\nu \in (-\bar{\nu}_1,\bar{\nu}_1)$, $n \in \mathbb{N}$ and $(P_{j_1},\ldots,P_{j_n}) \in (\mathcal{P}\setminus \mathcal{P}_1)^n$ we can associate a collision-free periodic solution $z_{((P_{j_1},\ldots,P_{j_n}),h,\nu)}$ of
\beq\label{problema ristretto}
\begin{cases}
\ddot{z}(t)+2 \nu i \dot{z}(t)= \nabla \Phi_\nu(z(t)) \\
\frac{1}{2}|\dot{z}(t)|^2-\Phi_\nu(z(t))=h,
\end{cases}
\eeq
which depends on $(P_{j_1},\ldots,P_{j_n})$ in the following way. There exist $\bar{R},\bar{\d}>0$ (depending on $h$ only) such that $z_{((P_{j_1},\ldots,P_{j_n}),h,\nu)}$ crosses $2n$ times within one period the circle $\partial B_{\bar{R}}(0)$, at times $(t_k)_{k=0,\dots,2n-1}$, and
\begin{itemize}
\item  in $(t_{2k},t_{2k+1})$ the solution stays outside $B_{\bar{R}}(0)$ and
$$
|z_{((P_{j_1},\ldots,P_{j_n}),h,\nu)}(t_{2k})-z_{((P_{j_{1}},\ldots,P_{j_n}),h,\nu)}(t_{2k+1})|<\bar{\d};
$$
\item in $(t_{2k+1},t_{2k+2})$ the solution lies inside $B_{\bar{R}}(0)$ and separates the centres according to the partition $P_{j_k}$.
\end{itemize}
\end{theorem}

The following picture represents the typical shape of a trajectory in the rotating frame of reference, moving with angular velocity $\nu$.
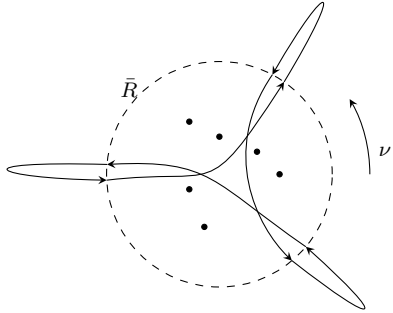
\begin{center}
\begin{tikzpicture}[>=stealth]
\draw[dashed] (0,0) circle (1.5cm);
\draw[font=\footnotesize] (130:1.5cm) node[anchor=east]{$\bar{R}$};
\draw[->] (2cm,0cm) arc (0:30:2cm);
\draw[font=\footnotesize] (2.2cm,0.3cm) node{$\nu$};
\filldraw (-0.4,0.7) circle (1pt)
          (0.5,0.3) circle (1pt)
          (0,0.5) circle (1pt)
          (-0.4,-0.2) circle (1pt)
          (-0.2,-0.7) circle (1pt)
          (0.8,0) circle (1pt) ;
%\draw[font=\footnotesize] (90:1.5cm) node[anchor=south west]{$x_0$};
\draw[->] (310:1.5cm)  .. controls (316:3cm) and (319:3cm)..(320:1.5cm);
\draw[->] (320:1.5cm)  .. controls (-0.1,-0.1) and (-0.1,0.2)..(175:1.5cm);
\draw[->] (175:1.5cm)  .. controls (177:3.5cm) and (181:3cm)..(183:1.5cm);
\draw[->] (183:1.5cm)  .. controls (0,0.1) and (-0.2,-0.4)..(55:1.5cm);
\draw[->] (55:1.5cm)  .. controls (57:3.1cm) and (61:3cm)..(62:1.5cm);
\draw[->] (62:1.5cm)  .. controls (0,0.3) and (0.5,-0.7)..(310:1.5cm);
%\draw[font=\footnotesize] (62:1.5cm) node[anchor=south east]{$x_0$};
\end{tikzpicture}
\end{center}

Note the analogy with Theorem 1.1 of \cite{SoTe}: if $\a=1$ and $N \geq 4$ we can easily find a condition on $(P_{j_1},\ldots,P_{j_n})$ in order to ensure that the periodic solution $z_{((P_{j_1},\ldots,P_{j_n}),h,0)}$ of the $N$-centre problem 
\[
\begin{cases}
\ddot{z}(t)= \n V(z(t)) \\
\frac{1}{2}|\dot{z}(t)|^2-V(z(t))=h
\end{cases}
\]
is collision-free; it is sufficient to impose that $P_{j_k} \in (\mathcal{P} \setminus \mathcal{P}_1)$ for every $k$. If $N=3$ then $\mathcal{P}=\mathcal{P}_1$, so that if in addition $\a=1$ we have to use a little trick: let 
\[
(P_1,P_1,P_2,P_3)=G_1, \qquad (P_2, P_2,P_3,P_1) = G_2,
\]
and let $\mathcal{G}:=\{G_1, G_2\}$. We will observe (Remark \ref{no coll N centri} below) that no composed sequence obtained by the juxtaposition of $G_1$ and $G_2$ satisfies the symmetry conditions of cases ($ii$)-$b$) or ($ii$)-$c$) of Theorem 1.1 in \cite{SoTe}; this implies that a solution of the $N$-centre problem associated to $(P_{k_1},\ldots,P_{k_{4n}}) \in \mathcal{G}^n \subset \mathcal{P}^{4n}$ is collision-free. Coming back to the rotating problem, it results

\begin{theorem}\label{esistenza di sol periodiche 3}
Replacing the assumption $N \geq 4$ in Theorem \ref{esistenza di soluzioni periodiche} with $N=3$, the same statement holds true replacing $(\mathcal{P} \setminus \mathcal{P}_1)^n$ with $\mathcal{G}^n$.
\end{theorem}

\noindent If $\a \neq 1$ this is not necessary, since in such a case $z_{((P_{j_1},\ldots,P_{j_n}),h,0)}$ was proved to be always collision-free. 

\begin{theorem}\label{esistenza di sol periodiche 2}
Replacing the assumptions $\a=1$ and $N \geq 4$ in Theorem \ref{esistenza di soluzioni periodiche} with $\a \in (1,2)$ and $N \geq 3$, the previous statement holds true, replacing the set $\mathcal{P} \setminus \mathcal{P}_1$ with $\mathcal{P}$.\end{theorem}

\begin{remark}
The assumption ``$|h|$ is sufficiently small" is substantial, as we can immediately realize observing that if $z$ is a solution of   \eqref{problema ristretto}, then the curve parametrized by $z$ in the configuration space has to be confined in $\{\Phi_\nu(z) \geq -h\}$. If $h<0$ becomes large in absolute value, we obtain a disconnected set, so that to find solutions exhibiting the behavior described in the previous statements becomes impossible.
\end{remark}

\subsection{Symbolic dynamics}

\noindent Similarly to Corollary 1.3 of \cite{SoTe}, as a consequence of Theorem \ref{esistenza di soluzioni periodiche}, \ref{esistenza di sol periodiche 3}, \ref{esistenza di sol periodiche 2}, we obtain the following result.

\begin{corollary}
Let $\a \in [1,2)$, $N \geq 3$, $m_1,\ldots,m_N \in \R^+$ and $c_1,\ldots,c_N \in \R^2$. Let $h \in (\bar h_1,0)$ and $ \nu \in (-\bar \nu_1(h),\bar \nu_1(h))$, where $\bar h_1$ and $\bar \nu_1(h)$ have been introduced in Theorem \ref{esistenza di soluzioni periodiche}, \ref{esistenza di sol periodiche 3}, \ref{esistenza di sol periodiche 2}.  There exists a subset $\Pi_{h,\nu}$ of the level set $\mathcal{U}_{h,\nu}$, a return map $\mathfrak{R}:\Pi_{h,\nu} \to \Pi_{h,\nu}$ for the dynamical system associated to equation \eqref{motion eq 2}, a set of symbols $\wh{\mathcal{P}}$ and a continuous and surjective map $\pi:\Pi_{h,\nu} \to \wh{\mathcal{P}}^{\Z}$, such that the diagram 
\[
 \xymatrix{
\Pi_{h,\nu} \ar[r]^{\mathfrak{R}} \ar[d]^\pi & \Pi_{h,\nu} \ar[d]^\pi \\
\wh{\mathcal{P}}^{\Z} \ar[r]^{T_r} &\wh{\mathcal{P}}^{\Z},
}
\]
commutes (here $T_r$ demotes the right shift in $\wh{\mathcal{P}}^{\Z}$); namely for every $h \in (\bar{h}_1,0)$ and $\nu \in (-\bar \nu_1(h),\bar \nu_1(h))$,  the restriction of the dynamical system associated to the rotating problem on the level set $\mathcal{U}_{h,\nu}$ has a symbolic dynamics.
\end{corollary}

\subsection{Plan of the paper}

We follow here the same general strategy already developed for proving Theorem 1.1 of \cite{SoTe}. 
In Section \ref{problema equivalente} we will perform a suitable rescaling in order to pass from problem \eqref{problema ristretto} to an equivalent problem where the parameter Jacobi constant will be replaced by the parameter given by the maximal distance of the centres from the origin. This leads to the study of a rotating problem with a rescaled potential
\beq\label{V_eps}
V_\eps(y)= \sum_{k=1}^N \frac{m_k}{|y-c_k'|^\a} \quad \text{where} \quad \max_{1 \leq k \leq N} |c_k'| = \eps,
\eeq
and a different angular velocity $\nu'$; we will be interested in solutions with Jacobi constant equal to $-1$. In this way, outside a ball or radius $R>\eps>0$, and for $|\nu'|$ sufficiently small, the equivalent problem
\beq\label{pb equiv}
\begin{cases}
\ddot{y}(t)+2 \nu' i \dot{y}(t)= \n \left(\frac{(\nu')^2}{2}|y|^2 + V_\eps(y) \right)\\
\frac{1}{2}|\dot{y}(t)|^2-\frac{(\nu')^2}{2}|y(t)|^2-V_\eps(y(t))=-1
\end{cases}
\eeq
is a small perturbation of the Kepler's problem with homogeneity degree $-\a<0$, $\a \in [1,2)$. This is why we will face the research of periodic solutions of \eqref{pb equiv} splitting the study of the dynamics outside/inside a ball $B_R(0)$ ($R$ will be conveniently chosen). As in \cite{SoTe}, outside $B_R(0)$ we will find arcs of solutions of \eqref{pb equiv} connecting two points $p_0,p_1 \in \pa B_R(0)$, provided their distance is sufficiently small, via perturbative techniques. With  respect to \cite{SoTe}, we have to take into account the new parameter $\nu'$, but the argument is substantially the same. \\
In section \ref{dinamica interna} we study the problem inside $B_R(0)$, trying again to follow the line of reasoning of \cite{SoTe}; we will search minimizers of the Jacobi type functional
\[
L_{h,\nu}:= \int_0^1 |\dot{u}| \sqrt{ \Phi_{\nu}(u)-1} + \frac{\nu}{\sqrt{2}} \int_0^1 \langle i u, \dot{u} \rangle
\]
under suitable constraints, in order to connect any pair $p_1, p_2 \in \pa B_R(0)$ with arcs of solution of \eqref{pb equiv} which separate the centres according to any prescribed partition in $\mathcal{P}$. The functional $L_{h,\nu}$, contrary to the classical Jacobi length, does not come from a Riemaniann structure but from a Finslerian one. A main consequence is the lack of reversibility of the problem, and this marks a significant difference in the argument to rule out the possibility of having collisions for its minimizers. The alternative "collision less" or "ejection-collision", valid for the $N$-centre problem, does not hold anymore. This is why we will need an "ad hoc" argument, which will be exposed in sections \ref{sezione collisioni} and \ref{dim con lemma}.\\
The collection of the outer and inner dynamics will be the object of section \ref{riduzione finito dim}. Assigned a sequence $(P_{j_1},\ldots,P_{j_n}) \in \mathcal{P}^n$ and $\eps$ and $\nu'$ sufficiently small, the aim will be the construction of a weak periodic solution $y_{((P_{j_1},\ldots,P_{j_n}),\eps,\nu')}$ of the restricted problem crossing $2n$ times within one period the circle $\partial B_R(0)$, at times $(t_k)_{k=0,\dots,2n-1}$, and
\begin{itemize}
\item  in $(t_{2k},t_{2k+1})$ the solution stays outside $B_R(0)$ and 
$$
|y_{((P_{j_1},\ldots,P_{j_n}),\eps,\nu)}(t_{2k})-y_{((P_{j_{1}},\ldots,P_{j_n}),\eps,\nu)}(t_{2k+1})|<\bar{\d}.
$$
\item in $(t_{2k+1},t_{2k+2})$ the solution lies inside $B_{\bar{R}}(0)$ and parametrizes an inner local minimizer of the functional $L_{-1,\nu'}$ which, up to collisions, separates the centres according to the partition $P_{j_k}$.
\end{itemize}
This will be achieved glueing the fixed ends trajectories found in sections \ref{dinamica esterna} and \ref{dinamica interna}, alternating outer and inner arcs. In order to obtain smooth junctions, we are going to use the variational argument already carried on in \cite{SoTe} with success. \\
Finally, in sections \ref{sezione collisioni} and \ref{dim con lemma}, we will complete the proof of Theorems \ref{esistenza di soluzioni periodiche}, \ref{esistenza di sol periodiche 3} and \ref{esistenza di sol periodiche 2}, providing sufficient conditions on the sequences $(P_{j_1},\ldots,P_{j_n})$ in order to have collision-free solutions; we will see that the minimizers of $L_{-1,\nu'}$ are weakly convergent in $H^1$, as $\nu' \to 0$, to the minimizers of $L_{-1,0}$, which is the classical Jacobi functional. Therefore we will exploit the description of the behaviour of such minimizers given in \cite{SoTe}.

\begin{remark}
If $\a=1$, the existence of periodic solutions to problem \eqref{problema ristretto} can be obtained by means of a perturbation argument in the following way: the Poincar\'e map associated to the $N$-center problem ($N \ge 3$) admits a compact hyperbolic invariant set of periodic points on any energy level $J_{h,0}$ with $h \ge 0$ (see Klein and Knauf \cite{KlKn}); the corresponding closed trajectories are global minimizers of the Jacobi length, and lies in a bounded region surrounding the centres. Due to the stability under perturbations of compact hyperbolic invariant sets, if $h<0$ and $|h|$ and $|\nu|$ are small enough, periodic solutions of problem \eqref{problema ristretto} still exist. \\
On the other hand, the results of \cite{SoTe} are not achieved through a perturbation argument from the case $h=0$. Actually, the periodic solutions we found tend, as $h \nearrow 0$, to a "concatenation" of parabolic unbounded orbits. In particular, since they were build by the gluing of constrained minimizers (near the centres) and perturbed Keplerian ellipses interacting with the boundary of the Hill's region (which, clearly, do not carry any hyperbolicity property), the previous discussion does not apply. This is why we have to adapt step by step the construction already carried on in \cite{SoTe}. Of course, compared with those obtained by Klein and Knauf, we obtain different periodic solutions yielding a new symbolic dynamics.
\end{remark}

\section{Preliminaries}\label{problema equivalente} 

Let us fix $N \geq 3$, $\a \in [1,2)$, $c_1,\ldots, c_N \in \R^2$ and $m_1, \ldots, m_N >0$, and let $M = \sum_{k=1}^N m_k$; we fix the origin in the centre of mass. In this section we prove that to find a periodic solution of the rotating problem \eqref{motion eq 2} with Jacobi constant $h<0$ is equivalent to find a periodic solution of a different rotating problem with Jacobi constant equal to $-1$. In this perspective the maximal distance of the centres from the origin replaces $h$ as parameter, and the angular velocity changes as well. To be precise one can easily prove:

\begin{proposition}\label{problema normalizzato}
Let $z \in \mathcal{C}^2\left((a,b);\R^2\right)$ be a classical solution of \eqref{motion eq 2} with Jacobi constant $h<0$. Then the function
\begin{equation}\label{eq:scaling}
y(t)= \left( -h \right)^{\frac{1}{\a}} z\left(\left(-h\right)^{-\frac{\a+2}{2\a}} t\right), \qquad t \in \left(\left(-h\right)^{\frac{\a+2}{2\a}}a,\left(-h\right)^{\frac{\a+2}{2\a}}b\right)
\end{equation}
is a solution of a rotating problem with
\beq\label{nuovi centri}
c_j'= \left(-h\right)^{\frac{1}{\a}} c_j, \quad j=1,\ldots,N \quad \text{and} \quad \nu'=\left(-h\right)^{-\frac{\a+2}{2\a}} \nu;
\eeq
the Jacobi constant of $y$ as solution of the new problem is $-1$. Conversely: let $y \in \mathcal{C}^2\left(\left(a',b'\right),\R^2\right)$ be a classical solution with Jacobi constant $-1$ of a rotating problem with initial configuration of the centres $\{c_j'\}$ and angular velocity $\nu'$. Let us set
\[
c_j= \left(-h\right)^{-\frac{1}{\a}} c_j', \quad j=1,\ldots,N \quad \text{and} \quad \nu=\left(-h\right)^{\frac{\a+2}{2\a}} \nu'.
\]
Then
\[
z(t)=\left(-h\right)^{-\frac{1}{\a}} y \left(\left(-h\right)^{\frac{\a+2}{2\a}}t\right), \qquad t \in \left(\left(-h\right)^{-\frac{\a+2}{2\a}} a', \left(-h\right)^{-\frac{\a+2}{2\a}}b'\right)
\]
is a classical solution of \eqref{motion eq 2} with Jacobi constant $h<0$.
\end{proposition}

\begin{corollary}\label{corol pb normalizzato}
For every $\eps>0$ and for every $\wt{\nu} \in \R$ there exist $\zeta_1(\eps)$ and $\zeta_2(\eps,\wt\nu) \in \R$ such that if $h=\zeta_1(\eps)$ and $\nu=\zeta_2(\eps,\wt{\nu})$ then
\[
\max_{1 \leq k \leq N} |c_k'|=\eps, \qquad \nu'= \wt{\nu}.
\]
The function $\zeta_1$ is strictly decreasing in $\eps$, the function $\zeta_2$ is strictly increasing both in $\eps$ and $\wt{\nu}$.
\end{corollary} 
%
%\begin{proof}
%Given $\eps>0$, from \eqref{nuovi centri} we obtain
%\[
%\zeta_1(\eps)=-\left(\frac{\eps}{\max_{1 \leq k \leq N} |c_k|} \right)^\a.
%\]
%Plugging this value of $h$ in the expression of $\nu'=(-h)^{-\frac{\a+2}{2\a}}\nu$ we obtain
%\[
%\nu'=\left(\frac{\max_{1 \leq k \leq N} |c_k|}{\eps}\right)^{\frac{\a+2}{2}} \nu.
%\]
%It is immediate to deduce
%\[
%\zeta_2(\eps,\wt{\nu})= \left(\frac{\eps}{\max_{1 \leq k \leq N} |c_k|}\right)^{\frac{\a+2}{2}} \wt{\nu}. \qedhere
%\]
%\end{proof}
%
\begin{remark}\label{h,nu--eps,nu'}
Problem \eqref{pb equiv} for $(\eps, \nu') \in (0,\bar{\eps}) \times (-\bar\nu',\bar\nu')$ is equivalent, through Proposition \ref{problema normalizzato} and Corollary \ref{corol pb normalizzato}, to  equation \eqref{motion eq 2} associated with Jacobi constant $h<0$ and angular velocity $\nu$ for $(h,\nu) \in (-\zeta_1(\bar\eps),0) \times (-\zeta_2(\bar\eps,\bar\nu), \zeta_2(\bar\eps,\bar\nu))$. Two corresponding solutions exhibit the same topological behaviour, as showed by equation \eqref{eq:scaling}. Note that more the Jacobi constant is small, more the admissible angular velocities have to be small.
\end{remark} 

Let us fix $\eps>0$, $\nu' \in \R$, and $K:= \overline{B_{R_2}(0)} \setminus B_{R_1}(0)$, with $R_2 > R_1>\eps$. In $K$ we can consider the new problem as a small perturbation of the $\a$-Kepler's problem, whose potential is
\[
V_0(y):= \frac{M}{\a |y|^\a} \qquad y \in \R^2 \setminus\{0\}.
\]
Indeed, setting
\[
\Phi_{\nu',\eps}(y):= \frac{(\nu')^2}{2}|y|^2+V_\eps(y),
\]
($V_\eps$ has been already defined in \eqref{V_eps}), it is not difficult to check that
\beq\label{perturbazione}
\|\Phi_{\nu',\eps}-V_0\|_{\mathcal{C}^1(K)}=o(\eps)+o(\nu') \qquad \text{for $\eps \to 0^+$, $\nu' \to 0$}.
\eeq

Let us observe that if $y$ is a solution of $\ddot{y}+2 \nu' i \dot{y}=\n \Phi_{\nu',\eps}(y)$ with Jacobi constant $-1$ over an interval $I \subset \R$, then
\[
\Phi_{\nu',\eps}(y(t)) \geq 1 \qquad \forall t \in I.
\]
To exploit the perturbative nature of the problem outside a ball $B_R(0)$, we have to check that, for $\eps>0$ sufficiently small and for $\nu'$ in a neighbourhood of $0$, there exists $R>0$ such that
\beq\label{eq condizione su R}
B_\eps(0) \subset B_R(0) \subset \left\lbrace y \in \R^2: \Phi_{\nu',\eps}(y) \geq 1\right\rbrace.
\eeq
Then, considering any compact set $B_R(0) \subset A \subset \{\Phi_{\nu',\eps}(y) \geq 1\}$, we will be able to use \eqref{perturbazione} in $A \setminus B_R(0)$.

\begin{proposition}\label{condizione su R}
Let $\eps>0$, $\nu'\in \R$. Let $R>0$ such that $\eps <R <\left(\frac{M}{\a}\right)^{1/\a}-\eps$. Then \eqref{eq condizione su R} holds true. There exists $\eps_1>0$ such that, for every $0<\eps<\eps_1$, this choice is possible.
\end{proposition}
%\begin{proof}
%Under our assumptions on $R$ for every $y \in \overline{B_R(0)}$ and for every $j$
%\[
%\Phi_{\nu',\eps}(y) \geq \sum_{k} \frac{m_k}{\a|y-c_k'|^\a} \geq \frac{M}{\a(R+\eps)^\a} > \frac{M}{\a \left(\left(\frac{M}{\a}\right)^{\frac{1}{\a}}\right)^\a}=1.
%\]
%There exists $\eps_1>0$ such that
%\[
%0<\eps<\eps_1 \Rightarrow \left(\frac{M}{\a}\right)^\frac{1}{\a}-\eps> \eps. \qedhere
%\]
%\end{proof}
\noindent Actually, we will make the further request $\eps <R/2 <R <\left(\frac{M}{\a}\right)^{1/\a}-\eps$.
which is satisfied for every $\eps \in (0, \eps_1 /2)$.  As in \cite{SoTe}, we select $R$ so that $\pa B_R(0)$ is the image of the circular solution of the $\a$-Kepler's problem with energy $-1$: 
\beq\label{scelta di R}
R:=\left(\frac{(2-\a)M}{2\a}\right)^{\frac{1}{\a}}.
\eeq
This is consistent with the previous restriction on $R$, if $\eps_1$ is sufficiently small (if this was not true, it is sufficient to replace $\eps_1$ with a smaller quantity).

\begin{remark}\label{maggiorazione in B_R}
For future convenience, note that for every $y \in \overline{B_R(0)}$ 
\beq\label{magg in B_R}
V_\eps(y)-1 \geq\frac{M}{\a\left(\left(\frac{(2-\a)M}{2\a}\right)^{\frac{1}{\a}}+\eps\right)^\a}-1\geq  \frac{M}{\a\left(\left(\frac{(2-\a)M}{2\a}\right)^{\frac{1}{\a}}+\eps_1\right)^\a}-1=:M_1>0,
\eeq
and hence $\Phi_{\nu',\eps}(y)-1 \geq M_1$. This value is independent on $\eps \in (0,\eps_1/2)$. From now on we will use $M_1$ to denote this positive constant. 
\end{remark}

\section{Outer dynamics}\label{dinamica esterna}

We are going to use a perturbative approach in order to find solutions of 
\beq\label{pb esterno}
\begin{cases}
\ddot{y}(t)+2\nu' i \dot{y}(t)=\n \Phi_{\nu',\eps}(y(t)) & t \in [0,T] \\
\frac{1}{2}|\dot{y}(t)|^2-\Phi_{\nu',\eps}(y(t))=-1 & t \in [0,T] \\
|y(t)|>R & t \in (0,T)\\
y(0)=p_0 \qquad y(T)=p_1
\end{cases}
\eeq
when the distance between $p_0,p_1 \in \pa B_R(0)$ is sufficiently small; $T$ has to be determined. To be precise we will prove the following proposition.

\begin{proposition}\label{teorema 0.1}
There exist $\d>0$, $\eps_2>0$ and $\nu'_1>0$ such that for every $(\eps,\nu') \in (0,\eps_2) \times (-\nu'_1,\nu'_1)$, for every $p_0,p_1 \in \pa B_R(0):|p_1-p_0| < 2\d$, there exist a unique solution $y_{\text{ext}}(\cdot\,;p_0,p_1;\eps,\nu')$ of \eqref{pb esterno} with $T=T_{\text{ext}}(p_0,p_1;\eps,\nu')>0$. This solution depends in a $\mathcal C^1$ way on the endpoints $p_0$ and $p_1$, and 
\beq\label{limitazione soluzioni esterne}
\begin{split}
& \max_{t \in [0,T_{\text{ext}}(p_0,p_1;\eps,\nu')]} |y_{\text{ext}}(t;p_0,p_1;\eps,\nu')| \leq 2\left(\frac{M}{\a}\right)^{\frac{1}{\a}} \\
& \max_{t \in [0,T_{\text{ext}}]} |\dot{y}_{\text{ext}}(t;p_0,p_1;\eps,\nu')| \leq  2\sqrt{2\left(-1+\frac{M}{\a R^\a}\right)} 
\end{split}
\eeq
for every $(p_0,p_1) \in \{ (p_0,p_1) \in (\pa B_R(0))^2: |p_0-p_1| <2\d \}$, $\eps \in (0,\eps_2)$ and $\nu' \in (-\nu_1',\nu_1')$.
\end{proposition}

\noindent We will follow the same line of reasoning of the proof of Theorem 3.1 of \cite{SoTe}, with the only difference that here we add the parameter $\nu'$. For the reader's convenience, we will review the main steps. For every $p_0=R \exp{\{i \t_0\}} \in \pa B_R(0)$, the unperturbed problem ($\eps=0$ and $\nu'=0$) is
\beq\label{PB_e}
\begin{cases}
\ddot{y}(t)=-M\frac{y(t)}{|y(t)|^{\a+2}}  & t \in [0,T] \\
\frac{1}{2}|\dot{y}(t)|^2-\frac{M}{\a |y(t)|^\a}=-1   & t \in [0,T]\\
|y(t)| > R   & t \in (0,T)\\
y(0)=p_0, \qquad y(T)=p_0.
\end{cases}
\eeq
Let us solve the Cauchy problem
\[
\begin{cases}
\ddot{y}(t)=-M\frac{y(t)}{|y(t)|^{\a+2}} \\
y(0)=p_0, & \dot{y}(0)= \sqrt{2\left(-1+\frac{M}{\a R^\a}\right)}\left(\frac{p_0}{R}\right).
\end{cases}
\]
The trajectory returns at the point $p_0$ after a certain time $\bar{T}>0$, having swept the portion of the rectilinear brake orbit of energy $-1$ starting from $p_0$ and lying in $\R^2 \setminus B_R(0)$. Our aim is to catch the behaviour of the solutions under small variations of the initial conditions. We consider
\beq\label{PC_e}
\begin{cases}
\ddot{y}(t)=-M\frac{y(t)}{|y(t)|^{\a+2}} \\
y(0)=p_0, \qquad \dot{y}(0)=  \dot{r}_0 e^{i \t_0}+ R \dot{\t}_0 i e^{i \t_0},
\end{cases}
\eeq
where $\dot{r}_0$ is assigned as function of $\dot{\t}_0$ by means of the energy integral. We denote as $y(\cdot\,;\t_0,\dot{\t}_0)$ the solution of \eqref{PC_e}. For the brake orbit $y\left(\cdot\,; \t_0,0\right)$, it results
\[
\t(t;\t_0,0)\equiv \t_0 \qquad \forall t \in [0,\bar{T}].
\]
We introduce $\psi:\Theta \times I \to \R^2$ as
\[
\psi(\dot{\t}_0, T):= y(T;\t_0,\dot{\t}_0),
\]
where $\Theta \times I \subset S^1 \times \R$ is a neighbourhood of $(0,\bar{T})$ on which $\psi$ is well defined. The following result is Lemma 3.2 of \cite{SoTe}, see the proof therein.

\begin{lemma}\label{lemma 0.1}
The Jacobian  of $\psi$ in $(0,\bar{T})$ is invertible.
\end{lemma}

Now we introduce the parameters $\eps$ and $\nu'$: let us define
\begin{align*}
\Psi : &\Theta \times I \times \pa B_R(0) \times \left[0,\frac{\eps_1}{2}\right) \times \R \to \R^2 \\
&(\dot{\t}_0,T,p_1,\eps,\nu') \mapsto y(T;\t_0,\dot{\t}_0;\eps,\nu')-p_1,
\end{align*}
where $y(\cdot \,;\t_0,\dot{\t}_0;\eps,\nu')$ is the solution of
\beq\label{PC_e2}
\begin{cases}
\ddot{y}(t)+2 \nu' i \dot{y}(t) =\nabla \Phi_{\nu',\eps}(y(t)) \\
y(0)=p_0, \qquad \dot{y}(0)= \dot{r}_{\nu',\eps} e^{i \t_0}+ R \dot{\t}_0 i e^{i \t_0},
\end{cases}
\eeq
and $\dot{r}_{\nu',\eps}$ is assigned as function of $\dot{\t}_0, \eps,\nu'$ by means of the Jacobi constant. The proof of the following statement is a straightforward generalization of the proof of Lemma 3.3 in \cite{SoTe}.

\begin{lemma}\label{lemma 0.2}
There exist $\d > 0$, $0<\eps_2<\eps_1/2$ and $\nu_1'>0$ such that for every $(\eps,\nu') \in (0,\eps_2) \times (-\nu'_1, \nu_1')$, for every $p_1 \in \pa B_R(0):|p_1-p_0| < 2\d$, there exists a unique solution $y(\cdot\,;\t_0,\dot{\t}_0;\eps,\nu')$ of \eqref{PC_e2} defined in $[0,T]$ for a certain $T>0$, and satisfying also \eqref{pb esterno}. Moreover, it is possible to choose $\delta$, $\eps_2$ and $\nu_1'$ independent on $p_0 \in \pa B_R(0)$.
\end{lemma}

\noindent Proposition \ref{teorema 0.1} follows. The solutions obtained are uniquely determined and depends in a smooth way on the ends $p_0$ and $p_1$, and on the parameters $\eps$ and $\nu'$ (by the implicit function theorem). Since a brake solution $y_{\text{br}}(\cdot)=y(\cdot\,;p_0,p_0;0,0)$ of the Kepler's problem is such that
\[
\max_{t \in [0,\bar{T}]} |y_{\text{br}}(t)|= \left(\frac{M}{\a}\right)^{\frac{1}{\a}} \quad \text{and} \quad \max_{t \in [0,\bar{T}]}|\dot{y}_{\text{br}}(t)|= \sqrt{2\left(-1+\frac{M}{\a R^\a}\right)},
\]
it is possible, if necessary, to replace $\eps_2$ and $\nu'_1$ with smaller quantities in such a way that \eqref{limitazione soluzioni esterne} is satisfied.

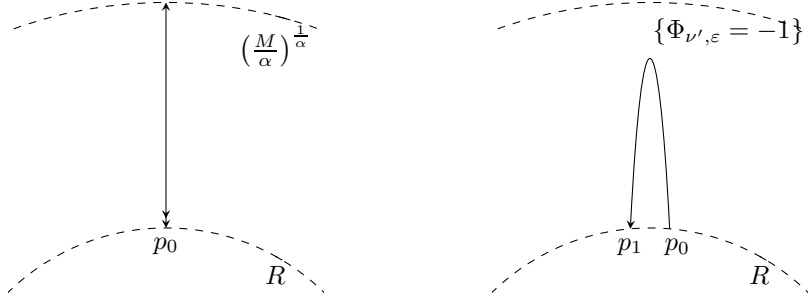
\begin{center}
\begin{tikzpicture}[scale=3,>=stealth]
\clip (-0.7,0.7) rectangle (0.7,2.3);
\draw[dashed]  (0,0) circle (1cm);
\draw (59:1cm)--(61:1cm) node[anchor=north]{$R$};
\draw (90:1cm) node[anchor=north]{$p_0$};
\draw[dashed] (0,0) circle (2cm);
\draw (74:2cm)--(76:2cm) node[anchor=north]{$\left(\frac{M}{\alpha}\right)^{\frac{1}{\alpha}}$};
\draw[<<->] (90:1cm)--(90:2cm) ;
\end{tikzpicture} $\hspace{2cm}$ \begin{tikzpicture}[scale=3,>=stealth]
\clip (-0.7,0.7) rectangle (0.7,2.3);
\draw[dashed]  (0,0) circle (1cm);
\draw (59:1cm)--(61:1cm) node[anchor=north]{$R$};
\draw (83:1cm) node[anchor=north]{$p_0$};
\draw (95:1cm) node[anchor=north]{$p_1$};
\draw[dashed] (0,0) circle (2cm);
\draw (80:2cm) node[anchor=north]{$\{\Phi_{\nu',\eps}=-1\}$};
\draw[->] (85:1 cm).. controls (89:2cm) and (91:2cm)..(95:1 cm);
\end{tikzpicture}
\end{center}
The picture represents the comparison between the rectilinear brake solution for the $\a$-Kepler problem and a ``close to brake" solution obtained for the perturbed problem with potential $\Phi_{\nu',\eps}$ via the implicit function theorem.

\begin{definition}\label{insieme sol esterne}
For any $\eps \in (0,\eps_2)$ we pose
\[
\mathcal{OS}_\eps:=\{ y_{\text{ext}}(\cdot\,;p_0,p_1; \eps,\nu'): \ p_0,p_1 \in \pa B_R(0), \ |\nu'|< \nu_1' \},
\]
i.e. $\mathcal{OS}_\eps$ is the set of the \emph{outer solutions} corresponding to a fixed value of $\eps$.
\end{definition}

\begin{lemma}\label{bound per i tempi esterni}
For every $\eps \in (0,\eps_2)$ there exist $C_1,C_2>0$ such that
\[
C_1 \leq T_{\text{ext}}(p_0,p_1;\eps,\nu') \leq C_2 \qquad \forall (p_0,p_1,\nu') \in \left(\pa B_R(0)\right)^2 \times (-\wt{\nu}',\wt{\nu}').
\]
Also, there exists $C_3>0$ such that
\[
\|y_{\text{ext}}(\cdot\,;p_0,p_1;\eps,\nu')\|_{H^1([0,T_{\text{ext}}(p_0,p_1;\eps,\nu')])} \leq C_3
\]
for every $(p_0,p_1,\nu') \in \left(\pa B_R(0)\right)^2 \times (-\wt{\nu}',\wt{\nu}')$.
\end{lemma}
\begin{proof}
The boundedness of $T_{\text{ext}}(p_0,p_1;\eps,\nu')$ is a consequence of the continuous dependence of the solutions with respect to variations of initial data. As far as the bound in the $H^1$ norm is concerned, we can use \eqref{limitazione soluzioni esterne} and the first part.
\end{proof}

\begin{remark}
We could make the boundedness properties described above uniform in $\eps$. But we will use this lemma in Sections \ref{riduzione finito dim}, \ref{sezione collisioni} and \ref{dim con lemma}, where $\eps$ will be fixed.
\end{remark}

\section{Inner dynamics}\label{dinamica interna}

In contrast with the previous one, this section is not a direct generalization of Section 4 of \cite{SoTe}; however, it is convenient to summarize the main ideas that we developed therein. Our goal was to find solutions of 
\beq\label{PI}
\begin{cases}
\ddot{y}(t)=\n V_\eps(y(t)) & t \in [0,T] \\
\frac{1}{2}|\dot{y}|^2-V_\eps(y(t))=-1 & t \in [0,T]\\
|y(t)|<R & t \in (0,T)\\
y(0)=p_1, \qquad y(T)=p_2.
\end{cases}
\eeq
satisfying particular topological requirements; $T$ was not determined a priori, while the energy was fixed to $-1$; hence, in order to give a variational formulation of \eqref{PI}, it was convenient to adopt the Maupertuis' principle rather then the minimal action principle. Let $[a,b] \subset \R$ and $p_1,p_2 \in \pa B_R(0)$, $p_1 = R \exp{\left\lbrace i \t_1\right\rbrace }$, $p_2= R \exp{\left\lbrace i\t_2\right\rbrace}$ (the case $p_1=p_2$ is admissible). We introduced the set of collision-free $H^1$ paths
\begin{multline}\label{H^_C.L.}
\wh{H}_{p_1 p_2}\left([a,b]\right):=\left\lbrace u \in H^1\left([a,b],\R^2\right): u(a)=p_1,\  u(b)=p_2, \right. \\
\left. u(t) \neq c_j \text{ for every $t\in [a,b]$, for every $j \in \left\lbrace 1,\ldots,N\right\rbrace$ }\right\rbrace, 
\end{multline}
the set of colliding $H^1$ functions
\begin{multline*}
\mathfrak{Coll}_{p_1 p_2}\left([a,b]\right):=\left\lbrace u \in  H^1\left([a,b],\R^2\right): u(a)=p_1,\  u(b)=p_2, \right. \\
\left. \exists t \in [a,b]: u(t) = c_j \text{ for some $j \in \left\lbrace 1,\ldots,N\right\rbrace$ }\right\rbrace ,
\end{multline*}
and their union
\[
H_{p_1 p_2}\left([a,b]\right)=\wh{H}_{p_1 p_2}\left([a,b]\right) \cup \mathfrak{Coll}_{p_1 p_2}\left([a,b]\right).
\]
Briefly, we will write $\wh{H}$, $\mathfrak{Coll}$ and $H$ when there will not be possibility of misunderstanding. Note that $H$ is the closure of $\wh{H}$ in the weak topology of $H^1$. A path $u \in \wh{H}$ can be characterized according to its winding number with respect to each centre. This number can be computed by artificially closing the path itself, in the following way:
\[
\Gamma(t):=\begin{cases}
                  \begin{cases}
                  u(t) & t \in [a,b]\\
                  R e^{i(t-b+\t_2)} & t \in (b,b+\t_1+2\pi-\t_2)
                  \end{cases} & \text{if $\t_1<\t_2$}\\
                  \, u(t) \qquad t \in [a,b] & \text{if $\t_1=\t_2$}\\
                  \begin{cases}
                  u(t) & t \in [a,b]\\
                  R e^{i(t-b+\t_2)} & t \in (b,b+\t_1-\t_2)
                  \end{cases} & \text{if $\t_1>\t_2$},
           \end{cases}
\]
i.e. if $p_1 \neq p_2$ we close the path $u$ with the arc of $\pa B_R(0)$ connecting $p_2$ and $p_1$ in counterclockwise sense. Then it is well defined the usual winding number $\textrm{Ind}\left(u([a,b]),c_j\right)$. Given $l=(l_1,\ldots,l_N) \in \Z^N$, a connected component of $\wh{H}$ is of the form
\[
\wh{\mathfrak{H}}_l^{p_1 p_2}([a,b]):=  \left\lbrace  u \in \wh{H}_{p_1 p_2}([a,b]): \textrm{Ind}\left(u([a,b]),c_j\right)=l_j \quad \forall j =1,\ldots, N\right\rbrace .
\]
We needed classes containing self-intersections-free paths, so that we considered $l \in \mathbb{Z}_2^N$ instead of $l \in \Z^N$, and set
\[
\wh{H}_l = \wh{H}_l^{p_1 p_2}([a,b]) := \left\lbrace  u \in \wh{H}_{p_1 p_2}([a,b]):  \text{Ind}\left(u([a,b]),c_j\right) \equiv l_j \mod 2 \quad \forall j=1,\ldots,N \right\rbrace;
\]
namely we collected together the components with winding numbers having the same parity with respect to each centre. We also assumed that
\beq\label{scelta di l 2}
\exists j,k \in \left\lbrace 1,\ldots,N\right\rbrace,\  j \neq k, \text{ such that } l_j \neq l_k \mod 2.
\eeq
In this way, each $u \in \wh{H}_l$ has to pass through the ball $B_\eps(0)$, and cannot be constant even if $p_1=p_2$. Actually we proved that the functions in $\wh{H}_l$ are uniformly non-constant, in the sense that there exists $C>0$ such that 
\[
\| \dot{u}\|_{L^2} \geq C \qquad \forall u \in \wh{H}_l.
\]
Furthermore, the constant $C$ can be chosen independently on $p_1$ and $p_2$ (see Lemma 5.2 of the quoted paper) and also on $l$ (the proof is the same). We said that $l \in \Z_2^N$ is a \emph{winding vector}, and we term $\mathfrak{I}^N:= \{ l \in \Z_2^N: \text{l satisfies \eqref{scelta di l 2}}\}$. In order to apply variational methods, we needed to consider $H_l=H_l^{p_1 p_2}([a,b])$, the closure of $\wh{H}_l$ with respect to the weak topology of $H^1$; of course, in $H_l$ there are collision-function. Since we searched functions whose images are in $B_R(0)$, we considered the subsets
\begin{gather*}
\wh{K}_l = \wh{K}_l^{p_1 p_2}([a,b]) :=\lbrace u \in \wh{H}_l: |u(t)| \leq R \ \forall t \in [a,b]\rbrace  \\
K_l = K_l^{p_1 p_2}([a,b]):=\left\lbrace u \in H_l: |u(t)| \leq R \ \forall t \in [a,b]\right\rbrace .
\end{gather*}
The set $K_l$ is weakly closed in $H^1$.
Recall the definition of the Maupertuis' functional associated to problem \eqref{PI}:
\beq\label{Maupertuis vecchio}
M_{-1}(u)=M_{-1}([a,b];u):= \frac{1}{2}\int_a^b |\dot{u}|^2 \int_a^b \left(V_\eps(u)-1\right);
\eeq
It is well known that solutions of the fixed energy problem given by the first two equations in \eqref{PI} are obtained as re-parametrizations of critical points of $M_{-1}$ at positive level in the space $\wh{H}$ (see, e.g. \cite{AmCZ}). It is also possible to consider re-parametrizations of critical points of the functional 
\beq\label{Jacobi vecchio}
L_{-1}(u)=L_{-1}([a,b];u):= \int_a^b \sqrt{\left(V_\eps(u)-1\right)|\dot{u}|^2},
\eeq
which is defined in the closure with respect to the weak topology of $H^1$ of
\[
H_{-1}=H_{-1}^{p_1 p_2}([a,b]):=\left\{ u \in H_{p_1 p_2}([a,b]):  V(u(t))>1, |\dot{u}(t)|>0 \text{ for a.e. $t \in [a,b]$}\right\}.
\]
Actually local minimizers of $M_{-1}$ are local minimizers of $L_{-1}$, and the converse is true up to a re-parameterization. The functional $L_{-1}$ has a useful geometric meaning, since for $u \in H_{-1}$ the value $L_{-1}(u)$ is the length of the curve parametrized by $u$ with respect to the Jacobi metric $g_{ij}(y)=\left(V_\eps(y)-1\right) \delta_{ij}$, where $\delta_{ij}$ is the Kronecker's delta; this metric makes the set $\{V_\eps(u)>1\}$ a Riemannian manifold. \\
Let us look at Theorem 4.12 of \cite{SoTe}. We proved that there exists $\eps_3>0$ such that for every $\eps \in (0,\eps_3)$, $p_1,p_2 \in \pa B_R(0)$ and $l \in \mathfrak{I}^N$ problem \eqref{PI} has a solution $\mathfrak{y}_{l}(\cdot\,;p_1,p_2;\eps) \in K_l^{p_1 p_2}([0,T])$ ($T=T(p_1,p_2;\eps;l)$) which is a re-parametrization of a local minimizer of the Maupertuis' functional $M_{-1}$ in $K_l^{p_1 p_2}([0,1])$, for some $T>0$. If $p_1=p_2$ and
\beq\label{105}
l_1=\cdots=l_{j-1}=l_{j+1}=\cdots=l_N \neq l_j \mod 2,
\eeq
then this solution can be an ejection-collision solution with a unique collision in $c_j$, otherwise it has to be self-intersection-free and collision-free.
The successive step consisted in the translation of Theorem 4.12 in the language of partitions. This is possible since if $u \in \wh{K}_l$ is self-intersection-free then it separates the centres in two different groups, which are determined by the particular choice of $l \in \mathfrak{I}^N$; namely, a self-intersection-free path in a class $\wh{K}_l$ induces a partition of the centres in two non-empty sets. Hence we could define the application $\mathcal{A}:\mathfrak{I}^N \to \mathcal{P}$ which associates to a winding vector
\[
l=(l_1,\ldots,l_N) \text{ with } \begin{cases}
                                  l_k\equiv 0 \mod 2 & k \in A_0 \subset \{1,\ldots,N\}\\
                                  l_k\equiv 1 \mod 2 & k \in A_1 \subset \{1,\ldots,N\}
                                  \end{cases}
\]
the partition
\[
\mathcal{A}(l):=\{\{c_k:l_k \in A_0\},\{c_k:l_k \in A_1\}\}.
\]
It is then natural to set
\[
\begin{split}
& \wh{K}_{P_j}=\wh{K}_{P_j}^{p_1 p_2}([a,b]):= \left\lbrace u \in \wh{K}_l^{p_1 p_2}([a,b]): l \in \mathcal{A}^{-1}(P_j) \right\rbrace,\\
& K_{P_j} = K_{P_j}^{p_1 p_2}([a,b]) :=\left\lbrace u \in K_l^{p_1 p_2}([a,b]):l \in \mathcal{A}^{-1}(P_j) \right\rbrace .
\end{split}
\]
In comparison with \cite{SoTe}, note that we don't require that a path in $K_{P_j}$ has no self-intersection; for the $N$-centre problem such a requirement was proved to be natural, in the sense that every minimizer of the Maupertuis' functional in $\wh{K}_l$ is necessarily self-intersection-free, unless it is an ejection-collision minimizer; for the rotating problem this is not necessarily true, therefore we drop this condition in the definition of $\wh{K}_{P_j}$.\\
From Theorem 4.12, we obtained, for every $\eps \in (0,\eps_3)$, $p_1,p_2 \in \pa B_R(0)$ and $P_j\in \mathcal{P}$, the existence of a solution $\mathfrak{y}_{P_j}(\cdot\,;p_1,p_2;\eps)$ of problem \eqref{PI}, which is a re-parametrization of a local minimizer of the Maupertuis' functional $M_{-1}$ in $K_{P_j}^{p_1 p_2}([0,1])$. If $p_1=p_2$ and $P_j\in \mathcal{P}_1$ then $\mathfrak{y}_{P_j}(\cdot\,;p_1,p_2;\eps)$ can be an ejection-collision solution with a unique collision in $c_i$, otherwise it is always collision-free (recall the definition of $\mathcal{P}_1$, equation \eqref{def di P_1}).

\vspace{1 em}

Let's come back to our "fixed Jacobi constant problem"
\beq\label{pb interno}
\begin{cases} 
\ddot{y}(t)+2 \nu' i \dot{y}(t)= \n \Phi_{\nu',\eps}(y(t)) & t \in [0,T]\\
\frac{1}{2}|\dot{y}(t)|^2-\Phi_{\nu',\eps}(y(t))=-1 & t \in [0,T]\\
|y(t)|<R & t \in [0,T]\\
y(0)=p_1 \qquad y(T)=p_2.
\end{cases}
\eeq
The variational formulation of \eqref{pb interno} will be the object of Subsection \ref{sez formulazione variazionale}. We will state the main result of this section in Subsection \ref{main result interno}.

\subsection{The variational formulation}\label{sez formulazione variazionale}

Let us consider a general problem of type
\beq\label{pb interno generale}
\begin{cases} 
\ddot{z}(t)+2 \nu i \dot{z}(t)= \n \Phi_\nu(z(t)) & t \in [0,T]\\
\frac{1}{2}|\dot{z}(t)|^2-\Phi_\nu(z(t))=h & t \in [0,T]\\
z(0)=p_1 \qquad z(T)=p_2.
\end{cases}
\eeq
with $T>0$ to be determined and $p_1, p_2 \in \R^2$. In order to solve it, we cannot use the Maupertuis' functional because it is suited for fixed energy problems. However, exploiting the existence of the Jacobi constant, we can study the Maupertuis-type functional
\[
M_{h,\nu}([a,b];u):=\sqrt{2}\left( \int_a^b |\dot{u}|^2\right)^{\frac{1}{2}}\left( \int_a^b \Phi_\nu(u)+h \right)^{\frac{1}{2}}+\nu \int_a^b \langle i u, \dot{u} \rangle.
\]
We will briefly write $M_{h,\nu}$ instead of $M_{h,\nu}([a,b];\cdot)$ when there is no possibility of misunderstanding. The domain of $M_{h,\nu}$ is the closure in the weak topology of $H^1$ of
\[
H_{h,\nu}^{p_1 p_2}([a,b]):= \left\{ u \in H_{p_1 p_2}(a,b]): \Phi_{\nu}(u(t)) > -h,|\dot{u}(t)|>0 \ \text{for a.e. $t \in[a,b]$} \right\}.
\]  
If  
\beq\label{livello positivo}
\sqrt{2}\left( \int_a^b |\dot{u}|^2\right)^{\frac{1}{2}}\left( \int_a^b \Phi_\nu(u)+h \right)^{\frac{1}{2}}>0,
\eeq
we can set
\beq\label{omega}
\o^2:= \frac{\int_a^b \Phi_\nu(u)+h}{\frac{1}{2}\int_a^b |\dot{u}|^2}>0
\eeq
and it makes sense to consider the re-parametrization $z(t)=u(\o t)$, defined in $[a/\o,b/\o]$. The functional $M_{h,\nu}$ is differentiable over $\wh{H} \cap \overline{H_{h,\nu}}^{\sigma(H^1, (H^1)^*)}$ (seen as an affine space on $H_0^1$). We will consider $[a,b]=[0,1]$ for the sake of simplicity.

\begin{theorem}\label{formulazione variazionale}
Let $u \in \wh{H}_{p_1 p_2}([0,1]) \cap \overline{(H_{h,\nu}^{p_1 p_2}([0,1])}^{\sigma(H^1, (H^1)^*)}$ be a critical point of $M_{h,\nu}$, i.e. $dM_{h,\nu}\left(u\right)[v]=0 $  for every $ v \in H_0^1\left([0,1]\right)$, and assume that \eqref{livello positivo} is satisfied. Let $\o$ be defined by \eqref{omega}. Then $z(t):=u(\o t)$ is a classical solution of \eqref{pb interno generale} with $T=1/\o$, while $u$ itself is a classical solution of
\beq\label{P_u}
\begin{cases}
\o^2 \ddot{u}(t)+2\nu \o i \dot{u}(t) = \n \Phi_\nu(u(t)) \qquad &t \in [0,1],\\
\frac{1}{2}|\dot{u}(t)|^2-\frac{\Phi(u(t))}{\o^2}=\frac{h}{\o^2} \qquad &t \in [0,1],\\
u(0)=p_1, \quad u(1)=p_2.
\end{cases}
\eeq
\end{theorem}

\begin{proof}
It is not difficult to check that if $dM_{h,\nu}(u)[v]=0$ for every $v \in H_0^1([0,1])$ then $z(t)=u(\o t)$ is a classical solution the first equation in \eqref{pb interno generale}. The Jacobi constant for $z$ reads
\[
\frac{1}{2} |\dot{z}(t)|^2-\Phi_\nu(z(t))=k  \quad \forall t \quad \Leftrightarrow \quad \frac{\o^2}{2}|\dot{u}(s)|^2-\Phi_\nu(u(s))=k \quad \forall s,
\]
where $k \in \R$. We deduce
\[
\o^2= \frac{\int_0^1 \Phi_\nu(u)+k}{\frac{1}{2}\int_0^1 |\dot{u}|^2};
\]
comparing with \eqref{omega}, we obtain $k=h$.
\end{proof}

\noindent The previous statement says that the functional $M_{h,\nu}$ plays, for problem \eqref{pb interno generale}, the role that the classical Maupertuis' functional $M_h$ plays for a fixed energy problem of type \eqref{PI}.  In order to apply variational methods it is worthwhile working in $H$ rather then in $\wh{H}$, since $\wh{H}$ is not weakly closed. As a consequence, it is not possible to rule out the occurrence of collisions from the beginning. This leads to the concept of weak solution for the problem \eqref{pb interno generale}. 
\begin{definition}
Let $u$ be a local minimizer of $M_{h,\nu}$ in $H_{h,\nu}^{p_1,p_2}([0,1])$ such that  \eqref{livello positivo} holds true, and let $\o$ be defined by \eqref{omega}. We say that $z(t)=u(\o t)$ is a \emph{weak solution} of \eqref{pb interno generale} in the time interval $[0,1/\o]$.
\end{definition}
If $z$ is a weak solution, we can define the collision set as:
\[
T_c(z):= \left\{t \in \left[0,\frac{1}{\o}\right]: z(t)=c_j \text{ for some $j=1,\ldots,N$}\right\}.
\]
It is not difficult to check that if $z$ is a weak solution and $(a,b) \subset [0,1] \setminus T_c(z)$, then $z$ is a classical solution of the restricted problem in $(a,b)$, with Jacobi constant $h$: indeed for every $\f \in \mathcal{C}^\infty_c(a,b)$ it results
\beq\label{eq53}
\left. \frac{d}{d\l} M_{h,\nu}(u+\l \f)\right|_{\l=0}=0.
\eeq
One can verify that the set $T_c(z)$ is discrete and finite, so that $z$ is a classical solution almost everywhere in $[0,1/\o]$.
On the other hand, a local minimizer in $K_l$ of $M_{h,\nu}$ does not satisfy the motion equation in every time interval $[c,d]$ such that $|u(t)|=R$ for every $t \in [c,d]$; indeed, in such a situation it is not true anymore that \eqref{eq53} holds true for every variation $\f \in \mathcal{C}^\infty_c([c,d])$. Nevertheless, the conservation of the Jacobi constant still holds true.
  
\begin{proposition}
If $u \in \overline{(H_{h,\nu}^{p_1 p_2}([0,1])}^{\sigma(H^1, (H^1)^*)}$ is a local minimizer of $M_{h,\nu}$, then
\[
\frac{1}{2}|\dot{u}(t)|^2-\frac{\Phi_{\nu}(u(t))}{\o^2}=\frac{h}{\o^2} \qquad \text{for a.e. $t \in [0,1]$}
\]
\end{proposition}

\begin{proof}
It is a consequence of the extremality of $u$ with respect to time re-parametrization keeping the ends fixed. For every $\f \in \mathcal{C}_c^\infty((0,1),\R)$, let us consider $u_\l (t):=u(t+\l \f(t))$. For $\l$ sufficiently small the function $t \mapsto t +\l \f(t)$ is increasing in $[0,1]$, so that in particular it is invertible; the minimality of $u$ implies
\[
\left. \frac{d}{d\l} M_{h,\nu}(u_\l)\right|_{\l=0}=0.
\]
%It is not difficult to prove that this implies the thesis.
\end{proof}

\begin{remark}\label{M_{h,nu}-M_h}
Note that, if $\nu=0$, the functional $M_{h,\nu}$ reduces to 
\[
M_{h,0}(u):=\sqrt{2}\left(\int_a^b |\dot{u}|^2\right)^{\frac{1}{2}}\left(\int_a^b \left(V(u)+h\right) \right)^{\frac{1}{2}} = 2 \sqrt{M_h(u)},
\]
where $M_h$ is the classical Maupertuis' functional of type \eqref{Maupertuis vecchio}. This reflects the perturbed nature of problem \eqref{pb interno}. Actually, due to the monotonicity of the square root for positive values of its argument it is immediate to deduce that $u$ is a (local) minimizer of $M_h$ at a positive level if and only if it is a (local) minimizer of $M_{h,0}$ such that \eqref{livello positivo} is satisfied. Therefore, if we work in a set in which $M_h$ is bounded below by a positive constant, it is equivalent to minimize $M_h$ or $M_{h,0}$. In particular, since in Lemma 4.16 of \cite{SoTe} we proved that for every $p_1,p_2 \in \pa B_R(0)$ and for every $l \in \mathfrak{I}^N$ there exists $C>0$ such that
\[
M_{-1}(u) \geq C>0  \qquad \forall u \in K_l^{p_1 p_2}([0,1]),
\]
the characterization of the minimizers of $M_{-1}$ in $K_l$ (and consequently also in $K_{P_j}$) described in Theorem 4.12 of \cite{SoTe} (or Corollary 4.14 of \cite{SoTe}) applies for the minimizers of $M_{-1,0}$; this will be crucial in Section \ref{sezione collisioni}.
\end{remark}

\noindent As announced in Section \ref{intro}, there is an analogue counterpart for the functional $L_h$. We introduce $L_{h,\nu}([a,b];\cdot):\overline{H_{h,\nu}}^{\sigma(H^1,(H^1)^*)} \to \R \cup \{+\infty\}$ as
\[
L_{h,\nu}([a,b];u):=\int_a^b \sqrt{\left(\Phi_\nu(u)+h\right)}|\dot{u}| + \frac{1}{\sqrt{2}} \nu \int_a^b \langle i u, \dot{u} \rangle.
\]
For $u \in H^1\left([a,b]\right)$ let us consider the following class of orientation-preserving re-parametrizations
\[
\Gamma_u:=\left\lbrace \left([c,d],f\right): f:[c,d] \to [a,b], \text{$f \in \mathcal{C}^1\left([c,d],\R\right)$ and increasing, such that $u \circ f \in H^1\left([c,d]\right)$}\right\rbrace .
\]
It is not difficult to check that $L_{h,\nu}$ is invariant under re-parametrizations of $\Gamma_u$. We point out that this is false if we consider re-parametrizations which do not preserve the orientation. In particular, differently from $L_h$, $L_{h,\nu}$ is not a length. It is possible to check that if $|\nu|$ is sufficiently small then
\[
\sqrt{\Phi_\nu(z)+h} |\dot{z}| + \nu \langle iu, \dot{u} \rangle
\]
is a Finsler function which makes the ``Hill's region" $\{\Phi_{\nu}(z)>-h\}$ a Finsler manifold. 

\begin{theorem}\label{teorema su L}
Let $u \in H_{h,\nu}^{p_1 p_2}([0,1]) \cap \wh{H}_{p_1 p_2}([0,1])$ be a non-constant critical point of $L_{h,\nu}$. 
Then there exist a re-parametrization $z$ of $u$ which is a classical solution of \eqref{pb interno generale} for some $T>0$.
\end{theorem}
\begin{proof}
We can adapt the proof of Theorem 4.5 of \cite{SoTe} with minor changes.
\end{proof}

The relationship between minimizers of $M_{h,\nu}$ and $L_{h,\nu}$ is given by the following statement.

\begin{proposition}\label{minimi M-L}
Let $u \in H_{h,\nu} \cap \wh{H}$ be a non-constant (local) minimizer of $M_{h,\nu}$ such that \eqref{livello positivo} holds true. Then $u$ is a (local) minimizer of $L_{h,\nu}$ in $H_{h,\nu} \cap \wh{H}$.\\
On the other hand, let $u \in H_{h,\nu} \cap \wh{H}$ be a non-constant (local) minimizer of $L_{h,\nu}$. Then, up to a re-parametrization, $u$ is a (local) minimizer of $M_{h,\nu}$ in $H_{h,\nu}\cap \wh{H}$ such that \eqref{livello positivo} holds true.  
\end{proposition}
\begin{proof}
Due to the H\"older inequality we have
\[
\sqrt{2}L_{h,\nu}(u) \leq M_{h,\nu}(u) \qquad \forall u \in H_{h,\nu} \cap \wh{H},
\]
with equality if and only if there exists $C>0$ such that 
\[
|\dot{u}(t)|^2=C\left(\Phi_\nu(u(t))-1 \right) \qquad \forall t \in [0,1].
\]
Now we can follow step by step the proofs of Proposition 4.6 and Proposition 4.7 of \cite{SoTe}.\end{proof}

\subsection{Existence of inner solutions}\label{main result interno}

The following result is a partial counterpart of Theorem 4.12 of \cite{SoTe}. 

\begin{proposition}\label{teo dinamica interna}
There exist $\eps_4 >0$ and $\nu_2'>0$ such that for every $(p_1,p_2,\eps,\nu',l) \in \left(\pa B_R(0)\right)^2 \times (0,\eps_4) \times (-\nu_2',\nu_2') \times \mathfrak{I}^N$, problem \eqref{pb interno} has a weak solution $y_{l}(\cdot\,;p_1,p_2;\eps,\nu') \in K_l^{p_1 p_2}([0,T])$ which is a re-parametrization of a local minimizer $u_l(\dot\,;p_1,p_2;\eps,\nu')$ of the Maupertuis'  functional $M_{-1,\nu'}$ in $K_l^{p_1 p_2}([0,1])$.
\end{proposition}

\noindent Before proceeding with the proof of Theorem \ref{teo dinamica interna}, we state the translation of this result in terms of partitions.

\begin{corollary}\label{dinamica interna partizioni}
For every $(p_1,p_2,\eps,\nu',P_j) \in \left(\pa B_R(0)\right)^2 \times (0,\eps_4) \times (-\nu_2',\nu_2') \times \mathcal{P}$, problem \eqref{pb interno} has a weak solution $y_{P_j}(\cdot\,;p_1,p_2;\eps,\nu') \in K_{P_j}^{p_1 p_2}([0,T])$ which is a re-parametrization of a local minimizer $u_{P_j}(\dot\,;p_1,p_2;\eps,\nu')$ of the Maupertuis'-type functional $M_{-1,\nu'}$ in $K_{P_j}^{p_1 p_2}([0,1])$.
\end{corollary}

We fix $[a,b]=[0,1]$ and the Jacobi constant to $-1$, so we will write $M_{\nu'}$ instead of $M_{-1,\nu'}$. Also, we fix $p_1,p_2 \in \pa B_R(0)$ and $l \in \mathfrak{I}^N$. 

\begin{remark}\label{remark dipendenza da nu}
In the statement of Theorem \ref{teo dinamica interna} the values $\eps_4$ and $\nu_2'$ depend neither on $p_1,p_2 \in \pa B_R(0)$, nor on $l \in \mathfrak{I}^N$. But here we fixed $p_1, p_2$ and $l$ before finding $\eps_4$ and $\nu_2'$. Actually, once we will find $\eps_4$ and $\nu_2'$, we will see that they are independent on the previous quantities.
\end{remark}

We aim at applying the direct methods of the calculus of variations in order to find a minimizer of $M_{\nu'}$ in $K_l$. Assuming that we can find such a minimizer $u_l(\cdot\,;p_1,p_2;\eps,\nu')$, in order to obtain a weak solution of \eqref{pb interno} we have to show that
\[
1) \  u_l(\cdot\,;p_1,p_2;\eps,\nu') \text{ satisfies \eqref{livello positivo}}, \qquad 2) \ 
|u_l(t;p_1,p_2;\eps,\nu')| < R  \quad  \forall t \in (0,1).
\]
Note that the first requirement is satisfied: for every $u \in \bigcup_{p_1,p_2,l}K_l^{p_1 p_2}([0,1])$, it results $|u| \leq R$; therefore we can use the bound of Remark \ref{maggiorazione in B_R}. We will discuss about the second condition after the minimization.

\begin{lemma}
The functional $M_{\nu'}$ is coercive in $K_l$.
\end{lemma}
\begin{proof}
Let $(u_n) \subset K_l$ such that $\| \dot{u}_n\|_{H^1} \to \infty$ for $n \to \infty$. Since $\|u_n\|_{L^2} \leq R$, necessarily $\|\dot{u}_n\|_{L^2} \to +\infty$ as $n \to \infty$. As $V_\eps(y)-1 \geq M_1>0$ in $B_R(0)$, 
\begin{align*}
M_{\nu'}(u_n) &\geq \sqrt{2} \| \dot{u}_n\|_{L^2}\left(M_1+ \frac{(\nu')^2}{2} \int_0^1 |u_n|^2 \right)^{\frac{1}{2}} - |\nu'| \int_0^1 |u_n| |\dot{u}_n| \\
&= \sqrt{2}  \|\dot{u}_n\|_{L^2} \left( \frac{|\nu'|}{\sqrt{2}} \|u_n\|_{L^2}+\l \right)-|\nu'| \|u_n\|_{L^2} \|\dot{u}_n\|_{L^2}
\end{align*}
for some $\l>0$. Hence $M_{\nu'}(u_n) \geq \sqrt{2} \l \|\dot{u}_n\|_{L^2}$.
\end{proof}

\begin{comment}
\begin{lemma}
There exists $\nu_2'>0$ such that for every $\nu' \in (-\nu_2',\nu_2')$ the functional $M_{\nu'}$ is coercive in $K_l$.
\end{lemma}

\begin{proof}
Let $(u_n) \subset K_l$ such that $\| \dot{u}_n\|_{H^1} \to \infty$ for $n \to \infty$. Since $\|u_n\|_{L^2} \leq R$, necessarily $\|\dot{u}_n\|_{L^2} \to +\infty$ as $n \to \infty$. Since $\Phi_{\nu',\eps}(y)-1 \geq M_1$ in $B_R(0)$
\[
M_{\nu'}(u_n) \geq \sqrt{2 M_1} \| \dot{u}_n\|_{L^2} - |\nu'| \int_0^1 |u_n| |\dot{u}_n| \geq \left(\sqrt{2 M_1}-|\nu'| R\right)\|\dot{u}_n\|_{L^2}.
\]
Setting $\nu_2':= \sqrt{2 M_1}/R$, for every $\nu' \in (-\nu_2',\nu_2')$ it results $\lim_n M_{\nu'}(u_n)=+\infty$. 
\end{proof}

Note that, since $M_1$ does not depend on $\eps \in (0,\eps_1/2)$, $p_1,p_2 \in \pa B_R(0)$, $l \in \Z_2^N$, the value $\nu_2'$ is independent on these quantities as well.
\end{comment}

\begin{lemma}\label{weak lower s-c di M_m}
The functional $M_{\nu'}$ is weakly lower semi-continuous in $K_l$.
\end{lemma}

\begin{proof}
Let $(u_n) \subset K_l$ such that $u_n \wc u$ weakly in $H^1$. It is by now standard the proof of 
\[
\left(\int_0^1 |\dot{u}|\right)^{\frac{1}{2}}\left(\int_0^1 \Phi_{\nu',\eps}(u)-1 \right)^ \frac{1}{2} \leq  \liminf_{n \to \infty}
\left(\int_0^1 |\dot{u}_n|\right)^{\frac{1}{2}}\left(\int_0^1 \Phi_{\nu',\eps}(u_n)-1 \right)^ \frac{1}{2},
\]
see for instance \cite{BaFeTe,Vethesis}. It remains to show that
\beq\label{eq8}
\nu' \int_0^1 \langle i u, \dot{u}\rangle \leq \liminf_{n \to \infty} \nu' \int_0^1 \langle i u_n,\dot{u}_n \rangle.
\eeq
The weak convergence of $u_n$ to $u$ implies that $u_n \to u$ uniformly in $[0,1]$ and $\dot{u}_n \wc \dot{u}$ weakly in $L^2$, as $n \to \infty$. We have
\[
\nu' \int_0^1 \langle i u_n, \dot{u}_n \rangle = \nu' \int_0^1 \langle i (u_n-u), \dot{u}_n \rangle + \nu' \int_0^1 \langle i u, \dot{u}_n \rangle.
\]
The first term tends to $0$ and the second term tends to $\nu' \int_0^1 \langle i u,\dot{u}\rangle$ as $n \to \infty$; \eqref{eq8} follows.
\end{proof}

\begin{remark}\label{remark 7}
The term $\nu \int_0^1 \langle i u,\dot{u} \rangle$ is not only weakly lower semi-continuous in $H^1$, but also continuous in the weak topology of $H^1$.
\end{remark}

\noindent Due to the coercivity and the weak lower semi-continuity of $M_{\nu'}$, we can apply the direct methods of the calculus of variations on the functional $M_{\nu'}$ in the weakly closed set $K_l$. For every $(\eps,\nu') \in (0,\eps_1/2) \times \R$, we obtain a minimizer $u_l(\cdot;p_1,p_2;\eps,\nu')$ for which \eqref{livello positivo} is satisfied. The following result concludes the proof of Proposition \ref{teo dinamica interna}.

\begin{lemma}
There are $\eps_4,\nu_2'>0$ such that for every 
$(p_1,p_2,\eps,\nu',l) \in \left(\pa B_R(0)\right)^2 \times (0,\eps_4) \times (-\nu_2',\nu_2') \times \mathfrak{I}^N$ the minimizer $u_l(\cdot\,;p_1,p_2;\eps,\nu')$ is such that
\[
|u_l(\cdot\,;p_1,p_2;\eps,\nu')| <R \qquad \forall t \in (0,1).
\]
\end{lemma}

\begin{proof}
We can follow the same line of reasoning which was used in \cite{SoTe} in order to prove Proposition 4.22. For the reader's convenience, we report here the ingredients of the proof.
Let us term
\[
T_R(u):=\left\lbrace t \in [0,1]: |u(t)|=R\right\rbrace , \quad T_{R/2}^+(u):=\left\lbrace t \in [0,1]: |u(t)| > \frac{R}{2}\right\rbrace
\]
A connected component of $T_R(u)$ is an interval (possibly a single point) $[t_1,t_2]$ with $t_1 \leq t_2$. It is possible to show that $u \in \mathcal{C}^1([0,1])$, and if $(a,b)$ is a connected component of $T_{R/2}^+(u) \setminus T_R(u)$, then $u|_{(a,b)}$ is of class $\mathcal{C}^2$ and is a solution of
\[
\o^2 \ddot{u}(t) + 2 \nu' \o i \dot{u}(t)=\n \Phi_{\nu',\eps}(u(t)), \quad \text{where} \quad \o^2:= \frac{\int_0^1 \left(\Phi_{\nu'\eps}(u)-1\right) }{\frac{1}{2}\int_0^1 |\dot{u}|^2}.
\]
Moreover, there are $\eps_4,\nu_2',\tau >0$ such that, if $(\eps,\nu') \in (0,\eps_4) \times (-\nu_2',\nu_2')$, then for every $t_3,t_4$ such that
\[
|u(t_3)| = R, \quad |u(t_4)|=\frac{R}{2}, \quad \frac{R}{2}<|u(t)|<R \quad \forall t \in \begin{cases} (t_3,t_4) & \text{if $t_3<t_4$}\\ (t_4,t_3) & \text{if $t_3>t_4$} \end{cases},
\]
there holds $|t_4-t_3|\leq \tau$. Neither $\eps_4$ nor $\nu_2'$ depend on $p_1,p_2$ or $l$. 
Let $[t_1, t_2]$ be a connected component of $T_R(u)$, let $(a,b)$ be a connected component of $T_{R/2}^+$ such that $[t_1,t_2] \subset (a,b)$. Let us consider $y(t):=u(\o t)$. Since $y \in \mathcal{C}^1\left((a/\o,b/\o)\right)$, it must lean against the circle $\left\lbrace y \in \R^2: |y|=R\right\rbrace$ with tangential velocity, and for every $\l>0$ there exists $t_5>t_2$ (or, if $t_2=1$, $t_5<t_1$, and in this case the following inequality has to be changed in obvious way) such that
\[
\left|y\left(\frac{t_5}{\o}\right)-Re^{i \t(t_2/\o)}\right| + \left|\dot{y}\left(\frac{t_5}{\o}\right)- R \dot{\t}\left(\frac{t_2}{\o}\right)i e^{i \t(t_2/\o)}\right| < \l.
\]
Thus, recalling that $R$ is the radius of the circular solution of energy $-1$ for the $\a$-Kepler's problem, the theorem of continuous dependence of the solutions with respect to the vector field and the initial data implies that $y$ cannot enter (or exit from) the ball $B_{R/2}(0)$ in time $\tau$, provided $\eps_4$ and $\nu_2'$ are sufficiently small (if this was not true, we can replace them with smaller quantities); this is in contradiction with the choice of $l$.
\end{proof}

\noindent In order to exploit the description of the behavior of the solution which we obtained for the $N$-centre problem in Theorem 4.12 of \cite{SoTe}, we will replace $\eps_4$ with $\min\{\eps_3,\eps_4\}$ (for the reader's convenience, we recall again that $\eps_3$ has been introduced in Theorem 4.12 of \cite{SoTe}).

\begin{definition}\label{def insieme soluzioni interne}
Let us fix \textbf{arbitrarily} $\nu_3' \in \left(0,\min\{\nu_2',\sqrt{2M_1}/R\}\right)$. For every $\eps \in (0,\eps_4)$ we term
\[
\mathcal{IM}_\eps:=\{ u_l(\cdot;p_1,p_2;\eps,\nu'): \ p_1,p_2 \in \pa B_R(0),\ l \in \Z_2^N, \ |\nu'|<\nu_3'\},
\]
the set of the \emph{inner minimizers} of $\{M_{\nu'}\}_{|\nu'| <\nu_3'}$ for a fixed value of $\eps$, and 
\[
\mathcal{IS}_\eps:=\{y_l(\cdot;p_1,p_2;\eps,\nu'): \ p_1,p_2 \in \pa B_R(0),\ l \in \Z_2^N, \ |\nu'| < \nu_3'\},
\]
the set of the corresponding \emph{inner solutions} for a fixed value of $\eps$.
\end{definition}

We conclude this section with a collection of boundedness properties for the functions of $\mathcal{IM}_\eps$.

\begin{proposition}\label{minimi limitati}
Let $\eps \in (0,\eps_4)$. There are $C_1,C_2,C_3,C_4,C_5>0$ such that
\begin{gather*}
C_1 \leq \inf_{u \in \mathcal{IM}_\eps} \|\dot{u}\|_{L^2} \leq \sup_{u \in \mathcal{IM}_\eps} \|\dot{u}\|_{L^2} \leq C_2,\\
C_3 \leq \inf_{u=u_l(\cdot\,;p_1,p_2,\eps,\nu')\in \mathcal{IM}_\eps} \int_0^1 \Phi_{\nu',\eps}(u)-1
\leq \sup_{u=u_l(\cdot\,;p_1,p_2,\eps,\nu') \in \mathcal{IM}_\eps} \int_0^1 \Phi_{\nu',\eps}(u)-1 \leq C_4,\\
\sup_{u=u_l(\cdot\,;p_1,p_2,\eps,\nu')} M_{\nu'}(u) \leq C_5.
\end{gather*}
\end{proposition}

\begin{remark}
Since $\sup\{ \|u\|_{L^2}: u \in \mathcal{IM}_\eps\} \leq R$, the set $\mathcal{IM}_\eps$ is bounded in the $H^1$ norm.
\end{remark}

\begin{proof}
Every $u \in \mathcal{IM}_{\eps}$ is of type $u_l(\cdot\,;p_1,p_2;\eps,\nu')$ for some $p_1,p_2 \in \pa B_R(0)$, $l \in \mathfrak{I}^N$, $\nu'\in (\nu_3',\nu_3')$. Since $ \mathfrak{I}^N$ is discrete and finite, we can prove the statement for a fixed $l$. In \cite{SoTe} we proved that the functions of $\bigcup_{p_1,p_2 \in \pa B_R(0)} K_l^{p_1 p_2}([0,1])$ are uniformly non-constant, which ensures the existence of $C_1$. Furthermore, as an immediate consequence of the estimate in Remark \ref{maggiorazione in B_R}, we obtain $C_3=M_1$. Now let us fix $\wt{p}_1, \wt{p}_2 \in \pa B_R(0)$; there exists $\wt{u} \in K_l^{\wt{p}_1 \wt{p}_2}([0,1])$ such that, for some $C_6>0$ and $\mu=\mu(\eps) \in (0,\eps)$, it results
\[
|\dot{\wt{u}}(t)|=C_6, \quad |\wt{u}(t)-c_j| \geq \mu(\eps) \quad \forall t \in [0,1], \forall j \in \{1,\ldots, N\}.
\]
For every $\nu' \in (-\nu_3',\nu_3')$ we have
\[
\int_0^1 \Phi_{\nu',\eps}(\wt{u}) = \int_0^1 \left( V_\eps(\wt{u})+\frac{(\nu')^2}{2} |\wt{u}|^2 \right) \leq \frac{M}{\a \mu^\a}+ \frac{(\nu'_3)^2}{2} R^2 =:C_7,
\]
where $C_7=C_7(\eps)$. Starting from this bound it is possible to obtain a uniform bound with respect to $p_1,p_2,\nu'$ for the level of the minimizers of $M_{\nu'}$. If $(p_1,p_2) \neq (\wt{p}_1,\wt{p}_2)$, we consider the path
\[
\wh{u}(t):=\begin{cases}
                 \zeta_{R}(3t;p_1,\wt{p}_1) & t \in [0,1/3]\\
                 \wt{u}(3t-1) & t \in (1/3,2/3]\\
                 \zeta_{R}\left(3t- 2;\wt{p}_2,p_2\right) & t \in (2/3,1],
                 \end{cases}
\]
where, for $p_*,p_{**} \in \pa B_R(0)$, $\zeta_R(\cdot\,;p_*,p_{**}):[0,1] \to \R^2$ parametrizes the shorter (in the Euclidean metric) arc of $\pa B_R(0)$ connecting $p_*$ and $p_{**}$ with constant velocity. As far as the velocity of $\zeta_R(\cdot\,;p_*,p_{**})$ is concerned, it is easy to see that it is uniformly bounded with respect to $p_*,p_{**}$. This, together with the assumptions on $\wt{u}$, implies that also the velocity of $\wh{u}$ is bounded in $[0,1]$, and
\[
M_{\nu'}(\wh{u})\leq C\left(\int_0^1 \Phi_{\nu',\eps}(\wt{u})-1+C\right)^{\frac{1}{2}}+ |\nu'|R C \leq C_5.
\]
Again, $C_5=C_5(\eps)>0$, while it does not depend on the ends $p_1$ and $p_2$ or on the parameter $\nu'$. Consequently, for the family of the minimizers there holds
\beq\label{eq9}
M_{\nu'}(u_l(\cdot\,;p_1,p_2;\eps,\nu')) \leq C_5 \qquad \forall p_1,p_2 \in \pa B_{R}(0), \ |\nu'| < \nu_3'.
\eeq
Using \eqref{magg in B_R}, we obtain 
\begin{align*}
\|\dot{u}_l(\cdot\,;p_1,p_2;\eps,\nu')\|_{L^2}  &\leq \frac{C_5-\nu' \int_0^1 \left\langle i u_l(\cdot\,;p_1,p_2;\eps,\nu'),\dot{u}_l(\cdot\,;p_1,p_2;\eps,\nu') \right\rangle }{\sqrt{2M_1}} \\
&\leq \frac{C_5 + |\nu'| R \|\dot{u}_l(\cdot\,;p_1,p_2;\eps,\nu')\|_{L^2}}{\sqrt{2 M_1}},
\end{align*}
for every $p_1,p_2 \in \pa B_R(0)$ and $|\nu'| < \nu_3'$. Now
\[
\left(1- \frac{|\nu'|R}{\sqrt{2 M_1}}\right)\|\dot{u}_l(\cdot\,;p_1,p_2;\eps,\nu')\|_{L^2} \leq \frac{C_5}{\sqrt{2 M_1}}.
\]
Since $|\nu'| < \nu_3' <\sqrt{2M_1}/R$, the coefficient on the left hand side is bounded below by a positive constant; therefore
\[
\|\dot{u}_l(\cdot\,;p_1,p_2;\eps,\nu')\|_{L^2} \leq  \frac{C_5}{\sqrt{2 M_1}} \left(1- \frac{|\nu_3'|R}{\sqrt{2 M_1}}\right)^{-1}=:C_2(\eps) \qquad \forall (p_1,p_2,\nu') \in \left(\pa B_R(0)\right)^2 \times (-\nu_3',\nu_3').
\]
It remains to find $C_4$; from \eqref{eq9}, using the existence of $C_1$, it follows 
\[
\left(\int_0^1 \Phi_{\nu',\eps}(u_l(\cdot\,;p_1,p_2;\eps,\nu'))-1\right)^{\frac{1}{2}}
 \leq \frac{C_5 + |\nu'| R \|\dot{u}_l(\cdot\,;p_1,p_2;\eps,\nu')\|_{L_2}}{\sqrt{2}\|\dot{u}_l(\cdot\,;p_1,p_2;\eps,\nu')\|_{L^2}}
\leq \frac{C_5}{\sqrt{2}C_1} + \frac{\nu_3'R}{\sqrt{2}}=:C_4^{\frac{1}{2}}. \qedhere
\]
\end{proof}

\begin{remark}
The fact that some constants depend on $\eps$ reflects the fact that more the Jacobi constant is small, more the admissible values of the angular velocity are small, see Remark \ref{h,nu--eps,nu'}. This is why we keep $\eps$ fixed, letting $\nu'$ vary, instead of considering both $\eps$ and $\nu'$ as parameters.
\end{remark}

\noindent We termed $[0,T_l(p_1,p_2;\eps,\nu')]$ as the time interval of $y_l(\cdot\,;p_1,p_2;\eps,\nu') \in \mathcal{IS}_\eps$. It results
\[
T_l(p_1,p_2;\eps,\nu')=\frac{1}{\o_l(p_1,p_2;\eps,\nu')}, \quad \text{where} \quad \o_l(p_1,p_2;\eps,\nu')=\frac{\int_0^1 \Phi_{\nu',\eps}(u_l(\cdot\,;p_1,p_2;\eps,\nu'))-1    }{\frac{1}{2} \|\dot{u}_l(\cdot\,;p_1,p_2;\eps,\nu')\|^2}.
\]

\begin{corollary}\label{bound tempi interni}
Let $\eps \in (0,\eps_4)$. There exist $C_1,C_2,C_3>0$ such that
\begin{gather*}
C_1 \leq T_l(p_1,p_2;\eps,\nu') \leq C_2 \\
\| y_l(\cdot\,;p_1,p_2;\eps,\nu')\|_{H^1([T_l(p_1,p_2;\eps,\nu')])} \leq C_3
\end{gather*}
for every $(p_0,p_1,\nu',l) \in \left(\pa B_R(0)\right)^2 \times (-\nu_3',\nu_3') \times \mathfrak{I}^N$.
\end{corollary}

\subsection{Forward normal neighborhoods}

In \cite{SoTe}, we exploited the geometric interpretation of $L$: it is the length in the Riemannian manifold $\{V_\eps(y) >-1\}$ endowed with the Jacobi metric. In particular in section 5 of the quoted paper we used classical results concerning the existence of totally normal and strongly convex neighborhoods (for the definitions, see \cite{DC}). Now we are not dealing with a length anymore, but with a Finsler function; so, something similar can be proven. The following is a known result, but since we cannot find a proper reference we give a sketch of the proof for completeness.

\begin{proposition}\label{existence forward neighbourhoods}
Let $\rho>0$ be small enough, in such a way that $B_\eps(0) \subset \overline{B_{R/2-\rho}(0)} \subset \overline{B_{R+\rho}(0)} \subset \{\Phi_{\nu',\eps}(y)>1\}$ and $R/2-\rho>\eps$. There exist $\eps_5 \in (0,\eps_4]$, $\nu_4' \in(0,\nu_3']$ and $\bar{r} \in (0,2\rho)$ such that if $\eps \in (0,\eps_5)$, $|\nu'| < \nu_4'$, $p_1,p_2 \in \overline{ B_{R}(0) \setminus B_{R/2}(0)}$ and $|p_1-p_2| \leq \bar r$ then there is a unique minimizer \\
$u_{\text{min}}(\cdot\,;p_1,p_2;\eps,\nu')$ of $M_{\nu'}$ in the set 
\[
\left\{ u \in H_{p_1 p_2}([0,1]): u(t) \in B_{R+\rho}(0) \setminus B_{R/2-\rho}(0) \ \forall t \right\}.
\]
Moreover, it depends in a $\mathcal{C}^1$ way on its ends and on the parameters $\eps$ and $\nu'$, and is the unique global minimizer of $M_{\nu'}$ in $H_{p_1 p_2}([0,1])$.
\end{proposition}

\begin{definition}
Let $\eps \in (0,\eps_5)$, $|\nu'|< \nu_4'$, and let us take $\rho>0$ as above; let $p \in \overline{B_R(0) \setminus B_{R/2}(0)}$. For every pair $p_1,p_2 \in \overline{B_{\bar r /2}}(p)$ there is a unique (up to a re-parametrization) local minimizer of $L_{\nu'}$ which starts from $p_1$ and arrives at $p_2$, depending smoothly on the ends. We will say that $B_{\bar{r}/2}(p)$ is a \emph{forward normal neighborhood} of $p$.
\end{definition}

\noindent Proposition \ref{existence forward neighbourhoods} says that every point of $\overline{B_R(0) \setminus B_{R/2}(0)}$ has a forward normal neighborhood; moreover, the set $B_{R+\rho}(0) \setminus B_{R/2-\rho}(0)$ is "convex", in the sense that the minimizers $u_{\text{min}}(\cdot\,;p_1,p_2;\eps,\nu')$ stay in it. \\
Forward normal neighborhoods plays the role of totally normal ones of a Riemaniann manifold, with the difference that, since our functional $L_{\nu'}$ is not invariant under orientation-reversing re-parameterizations, a minimizer of $L_{\nu'}$ in $H_{p_1 p_2}([0,1])$ could not be a minimizer of $L_{\nu'}$ in $H_{p_2 p_1}([0,1])$.\\
Actually for every $p \in \{\Phi_{\nu',\eps}(y)>1\}$ it is possible to prove the existence of a forward normal neighborhood, but due to the degeneracy of our Finsler function, which can become even negative if we are close to the boundary of the ``Hill's region", the radius of these neighborhood becomes smaller and smaller and tends to $0$ as $p$ approaches $\{\Phi_{\nu',\eps}(y)=1\}$.

\begin{proof}
Let $p_1,p_2 \in \overline{B_R(0) \setminus B_{R/2}(0)}$, $\eps \in (0,\eps_4)$, $\nu' \in (-\nu_4',\nu_4')$. The existence can be proved applying the direct methods of the calculus of variations. If $p_1=p_2$, observe that the minimizer is simply the constant function $p_1$.\\
Let $u_{\text{min}}(\cdot\,;p_1,p_2;\eps,\nu')$ be a minimizer in $H_{p_1 p_2}([0,1])$; there exists $\bar r>0$ such that if $|p_1-p_2| \leq \bar{r}$, then $u_{\min}(\cdot\,;p_1,p_2;\eps,\nu')$ is contained in $B_{R+\rho}(0) \setminus B_{R/2-\rho}(0)$: if not, there are sequences $(r_n) \subset \R^+$ and $((p_1^n,p_2^n)) \subset \overline{ B_{R}(0) \setminus B_{R/2}(0) } $ such that $|p_1^n-p_2^n|\leq r_n$ and $u_{\text{min}}(\cdot\,;p_1^n,p_2^n;\eps,\nu')$ touches $\pa \left(B_{R+\rho}(0) \setminus B_{R/2-\rho}(0) \right)$. But this is absurd, because if $r_n \to 0$ the minimizers tends to be constant functions in $\overline{ B_{R}(0) \setminus B_{R/2}(0)}$. The value $\r$ is independent on $\eps \in (0,\eps_4)$ and $|\nu'|<\nu_4'$. For the uniqueness and the $\mathcal{C}^1$ dependence, we consider the map 
\begin{gather*}
\left(\overline{ B_{R}(0) \setminus B_{R/2}(0) }\right)^2 \times (0,\eps_4) \times (-\nu_4',\nu_4') \times  H_{p_1 p_2}([0,1]) \to \left(H_{p_1p_2}([0,1]) \right)^* \\
 (p_1,p_2,\eps,\nu',u) \mapsto  dM_{\nu'}(u).
\end{gather*}
Let $\bar u$ be a minimizer of $M_{\nu'}$ in $H_{p_1 p_2}([0,1])$, whose image is contained in $B_{R+\rho}(0) \setminus B_{R/2-\rho}(0)$; an explicit computation shows that, if $|p_1-p_2|$ and $\nu'$ are sufficiently small, the second differential $d^2 M_{\nu'}(u)$ is positive definite, so that it is invertible. Thus, the implicit function theorem applies to give uniqueness and smooth dependence.
\end{proof}

\begin{remark}
In Section \ref{dinamica esterna} we prove that, if $p_1,p_2 \in \pa B_R(0)$ are sufficiently close together, we can find a ``close to brake" solution of problem \ref{pb esterno} which, of course, passes close to the boundary of the ``Hill's region" $\{\Phi_{\nu',\eps}(y)>1\}$. This is not in contradiction with the previous result, since an outer solution parametrizes a non-minimal critical point of $L_{\nu'}$.
\end{remark}

\section{A finite-dimensional reduction}\label{riduzione finito dim}

In this section we glue the fixed ends trajectories previously obtained, alternating outer and inner arcs in order to construct periodic orbits of the restricted problem \eqref{motion eq 2} in the whole plane. Since in this procedure we need smooth junctions, we are going to use a variational argument which is essentially the same we introduced in \cite{SoTe}. Let us set $\wt{\eps}:=\min\{\eps_2,\eps_5\}$, $\wt{\nu}':=\min\{\nu_1',\nu_4'\}$. The quantities $\eps_2$ and $\nu_1'$ have been introduced in Proposition \ref{teorema 0.1} (recall also the definition of $\d$ therein), while $\eps_5$ and $\nu_4'$ have been introduced in Proposition \ref{existence forward neighbourhoods}, respectively. This is the main result of this section.

\begin{proposition}\label{esistenza sol periodiche deboli}
There exist $\bar{\eps}, \bar{\nu}' >0$ such that for every $(\eps,\nu') \in (0,\bar{\eps}) \times (-\bar{\nu}',\bar{\nu}')$, for every $n \in \mathbb{N}$ and $(P_{j_1},\ldots,P_{j_n}) \in \mathcal{P}^n$ there exists a periodic weak solution $\gamma^{((P_{j_1},\ldots,P_{j_n}),\eps,\nu')}$ of problem \eqref{pb equiv}, which depends on $(P_{j_1},\ldots,P_{j_n})$ in the following way: the image of $\gamma^{((P_{j_1},\ldots,P_{j_n}),\eps,\nu')}$ crosses $2n$ times within one period the circle $\partial B_R(0)$, at times $(t_k)_{k=0,\dots,2n-1}$, and
\begin{itemize}
\item  in $(t_{2k},t_{2k+1})$ the solution stays outside $B_R(0)$ and %and  $x_{(P_{j_1},\ldots,P_{j_n})}(t_1), x_{(P_{j_1},\ldots,P_{j_n})}(t_2) \in \pa B_{\bar{R}}(0)$, %then
$$
|\gamma^{((P_{j_1},\ldots,P_{j_n}),\eps,\nu')}(t_{2k})-\gamma^{((P_{j_{1}},\ldots,P_{j_n}),\eps,\nu')}(t_{2k+1})    |<\d;
$$
\item in $(t_{2k+1},t_{2k+2})$ the solution lies inside $B_R(0)$, and, if it does not collide against any centre, then it separates them according to the partition $P_{j_k}$.
\end{itemize}
\end{proposition}

\vspace{1 em}

\noindent Let us fix $\eps \in (0,\wt{\eps})$, $|\nu'|< \wt{\nu}'$, $n \in \mathbb{N}$, $(P_{k_1},P_{k_2}, \ldots, P_{k_n})\in \mathcal{P}^n$. We define
\[
D=\left\lbrace (p_0, \ldots, p_{2n}) \in \left(\pa B_R(0)\right)^{2n+1}: |p_{2j+1} - p_{2j}| \leq \d
 \text{ for $j =0,\ldots, n-1$}, \  p_{2n}=p_0 \right\}.
\]
Let $(p_0, \ldots, p_{2n}) \in D$. For every $j \in \{0,\ldots,n-1\}$, we can apply Proposition \ref{teorema 0.1} to obtain an outer solution $y_{2j}(t):=y_{\text{ext}}(t;p_{2j},p_{2j+1};\eps,\nu')$ defined in $[0,T_{2j}]$, where $T_{2j}:= T_{\text{ext}}(p_{2j},p_{2j+1};\eps,\nu')$. We recall that $y_{2j}$ depends on $p_{2j}$ and $p_{2j+1}$ in a $\mathcal{C}^1$ manner. Also, from Corollary \ref{dinamica interna partizioni} we obtain an inner weak solution $y_{2j+1}(t):=y_{P_{k_{j+1}}} (t;p_{2j+1},p_{2j+2};\eps,\nu')$ defined in $[0,T_{2j+1}]$, where $T_{2j+1}:=T_{P_{k_{j+1}}}(p_{2j+1},p_{2j+2};\eps,\nu')$ (recall that $\nu_4'<\nu_3'$). Being $L_{\nu'}$ invariant under orientation-preserving re-parameterizations, $y_{2j+1}$ is a local minimizer of the functional $L_{\nu'}\left(\left[0,T_{2j+1}\right];\cdot\right)$. We point out that $y_{2j+1}$ could not be unique; however, if there is more then one minimizer of $L_{\nu'}$ in $K_{P_j}$, we can arbitrarily choose one of them.\\
We set $\mathfrak{T}_k:=\sum_{j=0}^k T_j$, $k=0,\ldots, 2n-1$, and
\beq\label{def di gamma}
\g_{(p_0,\ldots,p_{2n})}^{((P_{k_1},\ldots,P_{k_n}),\eps,\nu')}(s):= \begin{cases}
                            y_0(s) & s \in [0,\mathfrak{T}_0] \\
                            y_1(s-\mathfrak{T}_0) & s \in \left[\mathfrak{T}_0,\mathfrak{T}_1\right]\\
                            \vdots \\
                            \text{$\displaystyle{y_{2n-2}\left(s- \mathfrak{T}_{2n-3}\right)}$} &  \text{$\displaystyle{s \in \left[\mathfrak{T}_{2n-3},\mathfrak{T}_{2n-2}\right]}$} \\
                            \text{$\displaystyle{y_{2n-1}\left(s- \mathfrak{T}_{2n-2}\right)}$} & \text{$\displaystyle{s \in \left[\mathfrak{T}_{2n-2},\mathfrak{T}_{2n-1}\right]}$}.
                            \end{cases}
\eeq
The function $\g_{(p_0,\ldots,p_{2n})}^{((P_{k_1},\ldots,P_{k_n}),\eps,\nu')}$ is a piecewise differentiable $\mathfrak{T}_{2n-1}$-periodic function. It is a weak solution of the restricted problem \eqref{motion eq 2} with Jacobi constant $-1$ in $[0,\mathfrak{T}_{2n-1}] \setminus \left\{0,\mathfrak{T}_0,\ldots,\mathfrak{T}_{2n-1}\right\}$, but in general is not $\mathcal{C}^1$ in $\left\{0,\mathfrak{T}_0, \ldots, \mathfrak{T}_{2n-1}\right\}$; however, the right and left limits of the derivatives in these points are finite, so that it is in $H^1$. It is also possible that $\g_{(p_0,\ldots,p_{2n})}^{((P_{k_1},\ldots,P_{k_n}),\eps,\nu')}$ has collisions. Thanks to Lemma \ref{bound per i tempi esterni} and Corollary \ref{bound tempi interni}, we are sure that the time interval of $\g_{(p_0,\ldots,p_{2n})}^{((P_{k_1},\ldots,P_{k_n}),\eps,\nu')}$ is bounded above and bounded below, uniformly with respect to $(p_0,\ldots,p_{2n}) \in D$, by positive constants; therefore for every $(p_0,\ldots,p_{2n}) \in D$ the period of the associated function is neither trivial, nor infinite.

\vspace{1 em}

We introduce a function $F=F_{((P_{k_1},\ldots,P_{k_n}),\eps,\nu')}:D \to \R$ defined by
\[
F(p_0, \ldots, p_{2n}) := L_{\nu'}\left([0,\mathfrak{T}_{2n-1}]; \gamma_{(p_0,\ldots,p_{2n})}^{((P_{k_1},\ldots,P_{k_n}),\eps,\nu')}\right)\\
= \sum_{j=0}^{2n-1} L_{\nu'}\left([0,T_j]; y_j\right).
\]

\begin{proposition}\label{prop esistenza sol periodiche}
There exists $(\bar p_0, \ldots,\bar p_{2n}) \in D$ which minimizes $F$.
There exist $\bar{\eps},\bar{\nu}'>0$ such that, for every $(\eps,\nu') \in (0,\bar{\eps}) \times (-\bar{\nu}',\bar{\nu}')$, the associated function $\g_{(p_0,\ldots,p_{2n})}^{((P_{k_1},\ldots,P_{k_n}),\eps,\nu')}$ is a periodic weak solution of the restricted problem \eqref{pb equiv}. The values $\bar{\eps}$ and $\bar{\nu}'$ depends neither on $n$, nor on $(P_{k_1},\ldots,P_{k_n}) \in \mathcal{P}^n$.
\end{proposition}
\begin{remark}
Proposition \ref{esistenza sol periodiche deboli} is an immediate consequence of this statement.
\end{remark}

From now on, we will write $\g^{((P_{k_1},\ldots,P_{k_n}),\eps,\nu')}$ to denote the periodic weak solution associated to an arbitrarily chosen minimizer of $F_{((P_{k_1},\ldots,P_{k_n}),\eps,\nu')}$. 

We will reach the result through a series of lemmas. We will follow the same sketch already used in \cite{SoTe}, see also \cite{corr}.

\begin{lemma}\label{lem: continuity F}
The function $F$ is continuous, so that there exists a minimizer of $F$ in the compact set $D$.
\end{lemma}
\begin{proof}
Repeat the proof of step 1) of Theorem 5.3 of \cite{SoTe}.
\end{proof}

\begin{remark}
The main existence result of inner solutions, Proposition \ref{teo dinamica interna}, is stated in terms of winding vectors rather than in terms of partitions. Thus, it could seem reasonable to prescribe a finite sequence of winding vectors $(l_1,\dots,l_n) \in \Z_2^N$ and try to prove the existence of a periodic solution associated to this sequence in the same way as $\g^{((P_{k_1},\ldots,P_{k_n}),\eps,\nu')}$ is associated to $(P_{k_1},\dots,P_{k_n})$. This, clearly, would lead to a larger class of periodic solutions. But such a generalization does not seem possible, for the following reason. For the proof of Proposition \ref{prop esistenza sol periodiche} we consider variations of an inner minimizers with respect to its endpoints $p_1,p_2$; the function $\textrm{Ind}(u([a,b]),c_j)$ is not continuous in $u$ with respect to the uniform convergence topology if we let $p_1$ and $p_2$ vary on $\pa B_R(0)$, and this makes impossible to prove the continuity of a function like $F$. Note that the discontinuity occurs when $p_1=p_2$:
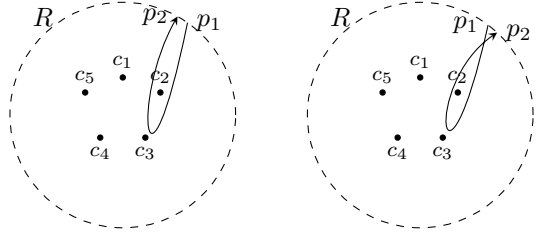
\begin{center}
\begin{tikzpicture}[>=stealth]
\draw[dashed] (0,0) circle (1.5cm);
\draw[->] (55:1.5cm).. controls (0.3,-1.5) and (0.1,0.2)..(61:1.5cm);
\filldraw[font=\footnotesize] (0.3,-0.3) circle (1pt) node[anchor=north]{$c_3$}
          (0.5,0.3) circle (1pt) node[anchor=south]{$c_2$}
          (0,0.5) circle (1pt) node[anchor=south]{$c_1$}
          (-0.3,-0.3) circle (1pt) node[anchor=north]{$c_4$}
          (-0.5,0.3) circle (1pt) node[anchor=south]{$c_5$};
\draw (135:1.5cm) node[anchor=south]{$R$};
\draw (55:1.5cm) node[anchor=west]{$p_1$};
\draw (61:1.5cm) node[anchor=east]{$p_2$};
\end{tikzpicture}
$\qquad$
\begin{tikzpicture}[>=stealth]
\draw[dashed] (0,0) circle (1.5cm);
\draw[->] (53:1.5cm).. controls (0.3,-1.5) and (0,0.4)..(47:1.5cm);
\filldraw[font=\footnotesize] (0.3,-0.3) circle (1pt) node[anchor=north]{$c_3$}
          (0.5,0.3) circle (1pt) node[anchor=south]{$c_2$}
          (0,0.5) circle (1pt) node[anchor=south]{$c_1$}
          (-0.3,-0.3) circle (1pt) node[anchor=north]{$c_4$}
          (-0.5,0.3) circle (1pt) node[anchor=south]{$c_5$};
\draw (135:1.5cm) node[anchor=south]{$R$};
\draw (47:1.5cm) node[anchor=west]{$p_2$};
\draw (53:1.5cm) node[anchor=east]{$p_1$};
\end{tikzpicture}
\end{center}
When $p_2$ moves continuously on $\pa B_R(0)$ and crosses $p_1$, although the two represented arcs remains ``close" in the uniform topology, the winding vector drastically changes, passing from $(1,0,1,1,1)$ to $(0,1,0,0,0)$ (recall that to compute the winding vector we close the arc with the portion of $\pa B_R(0)$ connecting $p_2$ with $p_1$ in counterclockwise sense). On the contrary, the partition which is determined by the inner arc does not change when $p_2$ crosses $p_1$. This makes possible to prove Lemma \ref{lem: continuity F} only when working with prescribed sequences of partitions, and not of winding vectors.
\end{remark}

Let $(\bar{p}_0,\ldots,\bar{p}_{2n})$ be a minimizer of $F$. We aim at showing that the minimality of $(\bar{p}_0,\ldots,\bar{p}_{2n})$ implies smoothness in the junction times for the associated periodic function $\gamma_{(\bar{p}_0,\ldots,\bar{p}_{2n})}^{((P_{k_1},\ldots,P_{k_n}),\eps,\nu')}$. In order to prove it, we would like to write explicitly the equation $\n F (\bar{p}_0,\ldots,\bar{p}_{2n}) = 0$. As we noticed in \cite{corr}, it is not evident that this can be done, because of the lack of uniqueness of inner minimizers of $M_{\nu'}$ in $K_{P_j}$: for this reason it is not immediate that an inner solution depends smoothly on its ends. In order to overcome the problem, we can use Proposition \ref{existence forward neighbourhoods}: for any $j \in \{0,\ldots,n-1\}$, we consider a forward normal neighborhood $U_{2j+1}$ of the point $\bar{p}_{2j+1}$. Let us choose $t_* \in (0,T_{2j+1})$ such that 
\[
\wt{p}_{2j+1}:=y_{2j+1}(t_*) \in  U_{2j+1},\quad |\wt{p}_{2j+1}|<R, \quad y\left([0,t_*]\right) \subset \left(B_R(0) \setminus B_{R/2}(0)\right);
\]
There exists a unique minimizer $\wh{y}(\cdot;\bar{p}_{2j+1},\wt{p}_{2j+1};\eps,\nu')$ of $M_{\nu'}$, and hence also of $L_{\nu'}$ (up to a re-parameterization), which connects $p_{2j+1}$ and $\wt{p}_{2j+1}$ in time $1$, and depends smoothly on its ends. For the uniqueness, $\wh{y}$ has to be a re-parametrization of $y_{2j+1}$. Note that if $p_{2j+1} \in \overline{U_{2j+1} \cap B_R(0)}$, then there is a unique minimizer $\wh{y}(\cdot;p_{2j+1},\wt{p}_{2j+1};\eps,\nu')$ of $M_{\nu'}$ which connects $p_{2j+1}$ and $\wt{p}_{2j+1}$. We will consider its re-parametrization $\wt{y}(\cdot\,;p_{2j+1},\wt{p}_{2j+1};\eps)$ such that
\[
\begin{cases}
\ddot{\wt{y}}(t)+2 \nu' i \dot{\wt{y}}(t)= \n \Phi_{\nu',\eps}(\wt{y}(t)) \\
\frac{1}{2}|\dot{\wt{y}}(t)|^2-\Phi_{\nu',\eps}(\wt{y}(t))=-1,
\end{cases}
\]
denoting by $[0,T(p_{2j+1},\wt{p}_{2j+1})]$ its domain. Due to the minimality of $\wh{y}(\cdot\,;p_{2j+1},\wt{p}_{2j+1};\eps,\nu')$ for $L_{\nu'}$, such a re-parametrization exists, see Theorem \ref{teorema su L}. In this way
\beq\label{oss 1}
\wt{y}(\cdot\,;\bar{p}_{2j+1},\wt{p}_{2j+1};\eps,\nu')\equiv y_{P_{k_{j+1}}}(\cdot\,;\bar{p}_{2j+1},\wt{p}_{2j+1};\eps,\nu')|_{[0,T(\bar{p}_{2j+1},\wt{p}_{2j+1})]}.
\eeq
Let $D_{2j+1}:=\{p_{2j+1} \in \left(\pa B_R(0) \cap \bar U_{2j+1} \right): |\bar{p}_{2j}-p_{2j+1}| \leq \d\}$. We define $G_{2j+1}:D_{2j+1} \to \R$ by
\begin{multline*}
G_{2j+1}(p_{2j+1}):=L\left([0,T(p_{2j+1})]; y_{\text{ext}}(\cdot\,;\bar{p}_{2j},p_{2j+1};\eps,\nu')\right)\\
+ L\left([0,T(p_{2j+1},\wt{p}_{2j+1})]; \wt{y}(\cdot\,;p_{2j+1},\wt{p}_{2j+1};\eps,\nu')\right),
\end{multline*}
where $T(p_{2j+1})$ denotes $T_{\text{ext}}(\bar{p}_{2j},p_{2j+1};\eps,\nu')$ (we will adopt this notation in this section). Of course, with minor changes we can also define a function $G_{2j}$, for every $j \in \{0,\ldots,2n\}$. Note that $G_k$ is continuous (for every $k$), since it is a sum of terms which are both continuous with respect to $p_k$. As a consequence, $G_k$ has a minimum. 

\begin{lemma}\label{localizzazione minimi F}
If $(\bar{p}_0,\ldots,\bar{p}_{2n})$ is a minimizer for $F$, then $\bar{p}_k$ is a minimizer for $G_k$.
\end{lemma}
\begin{proof}
The proof is the same of Lemma 1 of \cite{corr}.
\end{proof}

\noindent The main reason to pass from the study of $F$ to the study of the functions $G_k$ is that, in contrast with $F$, $G_k$ is differentiable for every $k$: let's think at $k=2j+1$; $L\left([0,T(p_{2j+1})]; y_{\text{ext}}(\cdot\,;\bar{p}_{2j},p_{2j+1};\eps,\nu')\right)$ depends smoothly on $p_{2j+1}$ for the differentiable dependence of outer solutions with respect to the ends, and $L\left([0,T(p_{2j+1},\wt{p}_{2j})]; \wt{y}(\cdot\,;p_{2j+1},\wt{p};\eps,\nu')\right)$ depends smoothly on $p_{2j+1}$ for Proposition \ref{existence forward neighbourhoods}. Therefore the minimality of $\bar{p}_{2j+1}$ implies that 
\[
\text{if }\bar{p}_{2j+1} \in D_{2j+1}^{\circ} \quad \Rightarrow \quad
\frac{\pa G_{2j+1}}{\pa p_{2j+1}}(\bar{p}_{2j+1})=0
\]
(the notation $D_{2j+1}^\circ$ denotes the inner of $D_{2j+1}$). This partial derivative is a linear operator from the tangent space $T_{\bar{p}_{2j+1}} (\pa B_R(0))$ into $\R$. In what follows we will show that, if $\eps$ and $\nu'$ are small enough, $\bar{p}_k \in D_k^{\circ}$ for every $k$, and that the stationarity conditions are nothing but regularity conditions for the functions 
\[
\zeta_{2j}(t):= \begin{cases} 
y_{P_{k_{j-1}}}(t+T_{2j-1}-T(\wt{p}_{2j},\bar p_{2j});\bar{p}_{2j-1},\bar{p}_{2j};\eps,\nu') &
 \text{if }t \in [0,T(\wt{p}_{2j},\bar{p}_{2j})] \\
y_{\text{ext}}(t-T(\wt{p}_{2j},\bar{p}_{2j});\bar{p}_{2j},\bar{p}_{2j+1};\eps,\nu') &
\text{if } t \in [T(\wt{p}_{2j},\bar{p}_{2j}),T(\wt{p}_{2j},\bar{p}_{2j})+T(\bar{p}_{2j+1})]
\end{cases}
\]
and
\[
\zeta_{2j+1}(t):= \begin{cases} y_{\text{ext}}(t;\bar{p}_{2j},\bar{p}_{2j+1};\eps,\nu') & \text{if } t \in [0,T(\bar{p}_{2j+1})]\\
y_{P_{k_{j+1}}}(t-T(\bar{p}_{2j+1});\bar{p}_{2j},\bar{p}_{2j+1};\eps,\nu') &
\text{if } t \in [T(\bar{p}_{2j+1}),T(\bar{p}_{2j+1})+T(\bar{p}_{2j+1},\wt{p}_{2j+1})].
\end{cases}
\]
Taking into account that $\zeta_k$ is (up to a time translation) the restriction of $\gamma^{((P_{k_1},\ldots,P_{k_n}),\eps,\nu')}$ on a neighbourhood of the junction time $\mathfrak{T}_{k-1}$, we obtain $\mathcal{C}^1$ regularity for $\gamma^{((P_{k_1},\ldots,P_{k_n}),\eps,\nu')}$ itself.
 
\begin{lemma}\label{calcolo delle derivate parziali}
For every $j=0,\ldots,n-1$, $p_{2j} \in D_{2j}$, and for every $\f \in T_{p_{2j}}(B_R(0))$, we have
\[
\frac{\pa G_{2j}}{\pa p_{2j}}(p_{2j})[\f] = \frac{1}{\sqrt{2}}\langle \dot{\wt{y}}(T(\wt{p}_{2j},p_{2j});\wt{p}_{2j},p_{2j};\eps,\nu') - \dot{y}_{\text{ext}}(0;p_{2j},\bar{p}_{2j+1};\eps,\nu'), \f \rangle .
\]
For every $j=0,\ldots,n-1$, $p_{2j+1} \in D_{2j+1}$, and for every $\f \in T_{p_{2j+1}}(B_R(0))$, we have
\[
\frac{\pa G_{2j+1}}{\pa p_{2j+1}}(p_{2j+1})[\f] = \frac{1}{\sqrt{2}}\langle \dot{y}_{\text{ext}}(T(p_{2j+1});\bar{p}_{2j},p_{2j+1};\eps,\nu') -\dot{\wt{y}}(0;p_{2j+1},\wt{p}_{2j+1};\eps,\nu'), \f \rangle .
\]
\end{lemma}
\begin{proof}
It is not restrictive to consider the derivative of $G_1$ to ease the notation. The same calculations work for the other cases. There holds
\beq\label{51*}
\frac{\pa G_1}{\pa p_1}\left(p_1\right) = \frac{\pa}{\pa p_1} L_{\nu'}\left([0,T(p_1)];y_{\text{ext}}(\cdot\,;\bar{p}_0,p_1;\eps,\nu')\right)
+\frac{\pa }{\pa p_1} L_{\nu'}\left([0,T(p_1,\wt{p}_1)]; \wt{y}(\cdot\,;p_1,\wt{p}_1;\eps,\nu')\right).
\eeq
Let us consider the first term in the right side, writing simply $y_0$ instead of \\
$y_{\text{ext}}(\cdot\,;\bar{p}_0,p_1;\eps,\nu')$; we consider $u_0(t)=y_0(T_0 t)$, defined in $[0,1]$. It results
\begin{multline*}
\frac{\pa}{\pa p_1} L_{\nu'}\left( [0,T(p_1)] ; y_0 \right)  = \frac{\pa}{\pa p_1} L_{\nu'}\left( [0,1] ; u_0 \right)  \\
 = \frac{1}{\sqrt{2}}\int_0^1 \left[\langle \frac{\dot{u}_0}{T_0}, \frac{d}{dt}\frac{\pa u_0}{\pa p_1}\rangle + \langle T_0\n \Phi_{\nu',\eps}(u_0),\frac{\pa u_0}{\pa p_1}\rangle  \right] + \frac{1}{\sqrt{2}} \nu' \int_0^1 \left( \langle i \frac{\pa u_0}{\pa p_1}, \dot{u}_0\rangle + \langle i u_0, \frac{d}{dt} \frac{\pa u_0}{\pa p_1} \right)\\
= \frac{1}{\sqrt{2}} \int_0^1 \langle -\frac{\ddot{u}_0}{T_0} - 2 \nu' i \dot{u}_0+ T_0 \n \Phi_{\nu',\eps}(u_0),\frac{\pa u_0}{\pa p_1}\rangle
  + \frac{1}{\sqrt{2}} \left[\langle \frac{\dot{u}_0(t)}{T_0}+\nu' i u_0(t),\frac{\pa u_0}{\pa p_1}(t)\rangle \right]_0^1\\
  = \frac{1}{\sqrt{2}}\left[\langle \dot{y}_0(t)+\nu' y_0(t),\frac{\pa y_0}{\pa p_1}(t)\rangle \right]_{0}^{T(p_1)}.
\end{multline*}
In the second equality we use the Jacobi constant for $y_0$, in the last one we use the fact that $y_0$ is a classical solution of the motion equation. \\
As in the step 3) of the proof of Theorem 5.3 of \cite{SoTe},  we can compute
\[
\frac{\pa }{\pa p_1} y_0(0)=0 \qquad \frac{\pa }{\pa p_1} y_0(T(p_1))= Id_{T_{p_1}(\pa B_R(0))}.
\]
Hence
\[
\frac{\pa}{\pa p_1} L_{\nu'}\left( [0,T(p_1)] ; y_0\right) \left[\f\right] =\frac{1}{\sqrt{2}}\left(\langle \dot{y}_0(T(p_1)), \f \rangle + \nu' \langle i p_1, \f \rangle \right).
\]
We can repeat the same computations for the second term in the right side of the \eqref{51*}, with minor changes: terming $\wt{y}_1=\wt{y}(\cdot\,;p_1,\wt{p}_1;\eps,\nu')$, we obtain
\[
\frac{\pa}{\pa p_1} L_{\nu'}\left( [0,T(p_1,\wt{p})_1] ; \wt{y}_1\right) \left[\f\right] = - \frac{1}{\sqrt{2}} \left(\langle \dot{\wt{y}}_1(0), \f \rangle + \nu' \langle i p_1, \f \rangle \right). \qedhere
\]
\end{proof}

\begin{lemma}\label{minimi interni}
There exist $\bar{\eps}>0$ and $\bar{\nu}'>0$ such that if $\eps \in (0,\bar{\eps})$ and $|\nu'| < \bar{\nu}'$ then
\[
\text{$\bar{p}_k$ minimizes $G_k$} \Rightarrow \bar{p}_k \in D_k^\circ \qquad \forall k.
\]
The values $\bar{\eps}$ and $\bar{\nu}'$ are independent on $(P_{k_1},\ldots,P_{k_n}) \in \mathcal{P}^n$.
\end{lemma}
\begin{proof}
Adapt the proof of Lemma 3 in \cite{corr}.
\end{proof}

\begin{proof}[Conclusion of the proof of Proposition \ref{prop esistenza sol periodiche}]
We can follow step 5) of the proof of Theorem 5.3 of \cite{SoTe} in order to check that each $\zeta_k$ is smooth. Recalling the construction of $\gamma^{((P_{k_1},\ldots,P_{k_n}),\eps,\nu')}$, the proof is complete.
\end{proof}

\section{Collision-free weak solutions}\label{sezione collisioni}

We will work with $\eps \in (0,\bar{\eps})$ which is fixed. The aim is to find a threshold $\bar{\nu}_{th}'(\eps)$ such that, if $|\nu'|<\bar{\nu}_{th}'(\eps)$, then $\gamma^{((P_{j_1},\ldots, P_{j_n}),\eps,\nu')}$ is collision-free. It is necessary to distinguish among:
\[
 1) \ \a=1 \text{ and } N \geq 4, \qquad 2) \ \a=1 \text{ and } N=3, \qquad 3) \ \a \in (1,2).
\]

\paragraph{1) $\a=1$ and $N \geq 4$.} We start by looking at Theorem 5.3 of \cite{SoTe}. Since $N \geq 4$, we have a simple way to choose $(P_{j_1},\ldots,P_{j_n})$ so that the weak solution $\g^{((P_{j_1},\ldots,P_{j_n}),\eps,0)}$ is a collision-free solution of the $N$-centre problem $\ddot{y}=\n V_\eps(y)$, with energy $-1$: it is sufficient to take $P_{j_k} \in \mathcal{P} \setminus \mathcal{P}_1$ for every $k=1,\ldots,n$. Indeed in such a situation the conditions ($ii$)-$b$) or ($ii$)-$c$) of the quoted statement cannot be satisfied. Note that if $N=3$ the set $\mathcal{P} \setminus \mathcal{P}_1$ is empty, and this is way that case deserves a different discussion. Now, let $\eps \in (0,\bar{\eps})$, $\nu' \in (-\bar{\nu}',\bar{\nu}')$, $n \in \N$ and $(P_{j_1},\ldots,P_{j_n}) \in (\mathcal{P} \setminus \mathcal{P}_1)^n$; let $(\bar{p}_0,\ldots, \bar{p}_{2n})$ be the minimizer of $F_{((P_{j_1},\ldots,P_{j_n}),\eps,\nu')}$ found in Proposition \ref{prop esistenza sol periodiche}, and let $\g^{((P_{j_1},\ldots,P_{j_n}),\eps,\nu')}$ be the corresponding periodic weak solution of \eqref{pb equiv}. Is it true that, for $\nu'$ sufficiently small, such a solution is still collision-free? The answer is affirmative: the idea is that if $\nu' \to 0$ the "minimizers" $\g^{((P_{j_1},\ldots,P_{j_n}),\eps,\nu')}$ are weakly convergent in $H^1$ to $\g^{((P_{j_1},\ldots,P_{j_n}),\eps,0)}$, which is collision-free. This is true in a local sense, and can be considered as a kind of Gamma-convergence argument.

\begin{con lemma}\label{minimi interni con variazione di nu}
Let $\eps \in (0,\bar{\eps})$, $P_{j} \in \mathcal{P}$, $((p_1^m,p_2^m)) \subset \left(\pa B_R(0)\right)^2$ and $(\nu_m') \subset (-\bar{\nu}',\bar{\nu}')$. Let $u_m=u_{P_j}(\cdot\,;p_1^m,p_2^m;\eps,\nu'_m)$ be a minimizer for the following variational problem:
\[
\min \left\{ M_{\nu_m'}(u): u \in K_{P_j}^{p_1^m p_2^m}([0,1]) \right\}.
\]
Assume that $(p_1^m,p_2^m) \to (\wt{p}_1,\wt{p}_2)$, $\nu_m' \to 0$, and $u_m \wc \wt{u}$ weakly in $H^1$. Then $\wt{u}$ is a minimizer for
\[
\min \left\{M_0(u):u \in K_{P_j}^{\wt{p}_1 \wt{p}_2} ([0,1]) \right\}.
\]
\end{con lemma}

We postpone the proof of this continuity lemma in the next section; now, as announced, we use it in order to prove the following proposition, which is the last step in the proof of Theorem \ref{esistenza di soluzioni periodiche} (recall Proposition \ref{problema normalizzato} and Remark \ref{h,nu--eps,nu'}).

\begin{proposition}\label{coll per N >=4}
Let $\a=1$ and $N \geq 4$. Let $\eps \in (0,\bar{\eps})$. There exists $\bar{\nu}_1'(\eps)$ such that for every $\nu'\in(-\bar{\nu}_1'(\eps),\bar{\nu}_1'(\eps))$, $n \in \N$ and $(P_{j_1},\ldots,P_{j_n}) \in (\mathcal{P} \setminus \mathcal{P}_1)^n$, the function $\g^{((P_{j_1},\ldots,P_{j_n}),\eps,\nu')}$ is collision-free.
\end{proposition}

\begin{proof}
Let $(P_{j_1},\ldots,P_{j_n}) \in( \mathcal{P} \setminus \mathcal{P}_1)^n$ and $\nu' \in (-\bar{\nu}',\bar{\nu}')$. The key observation is the following: when $\g^{((P_{j_1},\ldots,P_{j_n}),\eps,\nu')}$ stays inside $B_R(0)$, it coincides with a re-parameterization of an inner minimizer $u_{P_j}(\cdot\,;p_1,p_2;\eps,\nu')$, for some $p_1,p_2$ and $P_j$. Therefore the thesis follows if we show that there exist $\bar{\nu}_1'=\bar{\nu}_1'(\eps),\beta_1=\beta_1(\eps)>0$ such that 
\beq\label{eq14}
 \min_{k \in \{1,\ldots,N\}} \left( \min_{t \in [0,1]} |u_{P_j}(t;p_1,p_2;\eps,\nu')-c_k|\right) \geq \beta_1
\eeq
for every $(p_1,p_2,P_j,\nu') \in \left(\pa B_R(0)\right)^2 \times \left(\mathcal{P} \setminus \mathcal{P}_1\right) \times (-\bar{\nu}_1',\bar{\nu}_1')$. \\
Assume by contradiction that this claim is not true. Then there are $(\beta_m) \subset \R^+$, $(\nu'_m) \subset (-\bar{\nu}',\bar{\nu}')$, $\left((p_1^m,p_2^m)\right) \subset \left(\pa B_R(0)\right)^2$, $(P_j^m) \subset  \left(\mathcal{P} \setminus \mathcal{P}_1\right)$ and $(k_m) \subset \{1,\ldots,N\}$ such that $\beta_m \to 0$, $\nu_m'\to 0$ for $m \to \infty$, and 
\[
 \min_{t \in [0,1]} |u_{P_j^m}(t;p_1^m,p_2^m;\eps,\nu'_m)-c_{k_m}|=\beta_m \qquad \forall m.
\]
Since $\{1,\ldots,N\}$ and $\mathcal{P} \setminus \mathcal{P}_1$ are discrete and finite, we can assume $k_m=k$ and $P_j^m=P_j$ for every $m$. Also, since $\pa B_R(0)$ is compact, up to a subsequence $(p_1^m,p_2^m) \to (\wt{p}_1,\wt{p}_2) \in \pa B_R(0)$. We term $u_m=u_{P_j}(\cdot\,;p_1^m,p_2^m;\eps,\nu'_m)$. The set of the minimizers $\mathcal{IM}_\eps$ is bounded in the $H^1$ norm, therefore up to a subsequence $u_m \wc \wt{u} \in K^{\wt{p}_1 \wt{p}_2}_{P_j}([0,1])$ weakly in $H^1$ (and hence uniformly). In particular, the function $\wt{u}$ has at least one collision. The Continuity Lemma \ref{minimi interni con variazione di nu} implies that $\wt{u}$ is a collision minimizer of $M_0$ in $K^{\wt{p}_1 \wt{p}_2}_{P_j}([0,1])$; this is in contradiction with Theorem 4.12 of \cite{SoTe}, since $P_j \notin \mathcal{P}_1$ (recall Remark \ref{M_{h,nu}-M_h}).
\end{proof}

\begin{remark}
The Continuity Lemma permits to restrict the attention on a unique passage inside $B_R(0)$; in particular the argument is independent on $n$, which can be arbitrarily large.
\end{remark}

\paragraph{2) $\a=1$ and $N=3$.} This is the hardest part, since if we look at Theorem 5.3 of \cite{SoTe} we realize that it is not immediate to give conditions on $(P_{j_1},\ldots,P_{j_n})$ to obtain a collision-free periodic solution $\g^{((P_{j_1},\ldots,P_{j_n}),\eps,0)}$ for the fixed energy $N$-centre problem
\[
\begin{cases}
\ddot{y}(t)= \n V_\eps(y(t)) \\
\frac{1}{2}|\dot{y}(t)|^2-V_\eps(y(t))=-1.
\end{cases}
\]
In order to work with a set of symbols such that the corresponding solutions are collision-free, we introduced $\mathcal{G}$ (see section \ref{intro}); for every $n$ and for every $(P_{j_1},\ldots,P_{j_{4n}})\in \mathcal{G}^n$, the weak solution $\g^{((P_{j_1},\ldots,P_{j_{4n}}),\eps,0)}$ of the $N$-centre problem is actually a classical solution, because no composed sequence of elements of $\mathcal{G}$ has the reflection symmetry which characterizes a collision trajectory (see the following Remark \ref{no coll N centri}). For $\eps \in (0,\bar{\eps})$, we aim at showing that, if $|\nu'|$ is sufficiently small, for every $n \in \N$ and $(P_{j_1},\ldots,P_{j_{4n}}) \in \mathcal{G}^n$ the function $\g^{((P_{j_1},\ldots,P_{j_{4n}}),\eps,\nu')}$ is still collision-free. The idea for the proof is exactly the same which we have already used in point 1). Unfortunately, while therein we can simply restrict our attention to the behaviour of any inner minimizer (that is a local argument), here this approach does not work. Indeed, for every $P_j \in \mathcal{P}$ and $p_1 \in \pa B_R(0)$ it is possible that a minimizer of $M_0$ in $K_{P_j}^{p_1 p_1}([0,1])$ has collisions. Therefore we have to use an argument which is local, ``but not too much''.

\begin{remark}\label{no coll N centri}
A possible way to check that there aren't collisions for solutions to the $3$-centre problem associated to sequences of partitions of $\mathcal{G}$ is the following. Let $\gamma^{((P_{k_1},\ldots,P_{k_{4n}}),\eps,0)}$ be the periodic solution of the $N$-centre problem found in Theorem 5.3 of \cite{SoTe}. Writing $(P_{k_1},\ldots,P_{k_{4n}}) \in \mathcal{G}^n$ as an infinite periodic sequence, a group of 5 consecutive partitions is one of the following:
\beq\label{5 simb}
\begin{split}
P_1 P_1 P_2 P_3 P_1 \quad P_1 P_1 P_2 P_3 P_2  \quad P_1 P_2 P_3 P_1 P_1  \quad P_1 P_2 P_3 P_2 P_2  \quad P_2 P_3 P_1 P_1 P_2  \quad P_2 P_3 P_2 P_2 P_3 \\
P_3 P_1 P_1 P_2 P_3  \quad P_3 P_2 P_2 P_3 P_1  \quad P_2 P_2 P_3 P_1 P_1  \quad P_2 P_2 P_3 P_1 P_2  \quad P_2 P_3 P_1 P_1 P_1 \\
P_2 P_3 P_1 P_2 P_2   \quad P_3 P_1 P_1 P_1 P_2  \quad P_3 P_1 P_2 P_2 P_3  \quad P_1 P_1 P_1 P_2 P_3  \quad P_1 P_2 P_2 P_3 P_1.
\end{split}
\eeq
Assume that the considered solution has a collision with the centre $c_1$. According to the periodicity of $\gamma^{((P_{k_1},\ldots,P_{k_{4n}}),\eps,0)}$ and recalling that any collision solution is a collision-ejection solution, this means that there exists a group of five consecutive partitions $(P_{k_1},\ldots,P_{k_5})$ in \eqref{5 simb} such that
\begin{itemize}
\item $P_{k_3}=P_1$;
\item $P_{k_1}=P_{k_5}$ and $P_{k_2}=P_{k_4}$.
\end{itemize}
It is immediate to check that none of the groups in \eqref{5 simb} satisfies both the requirements. Analogously, it is possible to check that $\gamma^{((P_{k_1},\ldots,P_{k_{4n}}),\eps,0)}$ does not collide against $c_2$ or $c_3$.
\end{remark}

\noindent We collect the possible groups of 5 consecutive partitions in \eqref{5 simb} in a set $\wt{\mathcal{P}}^5 \subset \mathcal{P}$. Let us fix $\eps \in (0,\bar{\eps})$, $p_1,p_{10} \in \pa B_R(0)$, $\left(P_{k_1},\ldots,P_{k_5}\right) \in \wt{\mathcal{P}}^5$,  $\nu' \in (-\bar{\nu},\bar{\nu})$. Let 
\[
B:=\{(p_2, \ldots, p_9) \in (\pa B_R(0))^8: |p_{2j}-p_{2j+1}| \leq \d, \  j=1,\ldots,4\}.
\]
As we associated to each point of $D$ a periodic function, to each point of $B$ we can associate a (non-periodic) function in the following way. For each $j=1,\ldots,4$ we can connect $p_{2j}$ and $p_{2j+1}$ with an outer solution $y_{2j}=y_{\text{ext}}(\cdot\,;p_{2j},p_{2j+1};\eps,\nu')$ of \eqref{pb esterno}; for each $j=0,\ldots,4$ we can connect $p_{2j+1}$ and $p_{2j+2}$ with an inner solution $y_{2j+1}=y_{P_{k_{j+1}}}(\cdot\,;p_{2j+1},p_{2j+2};\eps,\nu')$ of \eqref{pb interno}. We set $t_1:=0$, $t_k:= \sum_{j=1}^{k-1} T_j$ for $k=2,\ldots,10$, where $[0,T_j]$ is the time interval of $y_j$. We define 
\beq\label{eq15}
\sigma_{(p_2,\ldots,p_9)}^{((p_1,p_{10}),(P_{k_1},\ldots, P_{k_5}),\eps,\nu')}(t):=\begin{cases}
y_1(t) & t \in [t_1,t_2]\\
y_2(t-t_2) & t \in [t_2,t_3]\\
\vdots\\
y_9(t-t_9) & t \in [t_9,t_{10}].
\end{cases}
\eeq
By the definition $\sigma_{(p_2,\ldots,p_9)}^{((p_1,p_{10}),(P_{k_1},\ldots, P_{k_5}),\eps,\nu')}(t_k)=p_k$. We introduce a function \\$\mathfrak{F}_{((p_1,p_{10}),(P_{k_1},\ldots,P_{k_5}),\eps,\nu')}:B \to \R$ as
\[
\mathfrak{F}_{((p_1,p_{10}),(P_{k_1},\ldots,P_{k_5}),\eps,\nu')}(p_2,\ldots,p_9):=L_{\nu'}\left([0, t_{10}];\sigma_{(p_2,\ldots,p_9)}^{((p_1,p_{10}),(P_{k_1},\ldots, P_{k_5}),\eps,\nu')}\right).
\]
Note the analogy between the definition of $\mathfrak{F}=\mathfrak{F}_{((p_1,p_{10}),(P_{k_1},\ldots, P_{k_5}),\eps,\nu')}$ and $F=F_{((P_{k_1},\ldots,P_{k_n}),\eps,\nu')}$. The function $\mathfrak{F}$ is continuous on the compact set $B$ (apply the same proof already used for the continuity of $F$), therefore it has a minimum. We denote by $\sigma^{((p_1,p_{10}),(P_{k_1},\ldots, P_{k_5}),\eps,\nu')}$ the glued function associated to an arbitrarily chosen minimizer. \\
Let $(P_{k_1},\ldots,P_{k_{4n}}) \in \mathcal{G}^n$. The following Lemma relates the minimality properties of $F$ and of $\mathfrak{F}$; in what follows the indexes have to be considered by periodicity: for instance writing $2j+5$ we mean $2j+5 \mod 8n$.

\begin{lemma}\label{minimi F e minimi F}
Let $(\bar{p}_0,\ldots,\bar{p}_{8n})\in D$ be a minimizer of $F_{((P_{k_1},\ldots,P_{k_{4n}}),\eps,\nu')}$. Then, for every $j=0,\ldots,4n-1$, the point $(\bar{p}_{2j+2},\ldots,\bar{p}_{2j+9}) \in B$ is a minimizer of \\
$\mathfrak{F}_{((\bar{p}_{2j+1},\bar{p}_{2j+10}),(P_{k_{j+1}},\ldots,P_{k_{j+5}}),\eps,\nu')}$. In particular
\[
\gamma_{((P_{k_{j+1}},\ldots,P_{k_{j+5}}),\eps,\nu')}|_{[\mathfrak{T}_{2j},\mathfrak{T}_{2j+10}]} \equiv \sigma^{((\bar{p}_{2j+1},\bar{p}_{2j+10}),(P_{k_{j+1}},\ldots, P_{k_{j+5}}),\eps,\nu')}.
\]
\end{lemma}
\begin{proof}
It is an immediate consequence of the additivity of the functional $L_{\nu'}$.
\end{proof}

\noindent As a consequence, the following statement can be proved applying the same argument already explained in Remark \ref{no coll N centri}.

\begin{lemma}\label{no coll 3-centri}
Let $\eps \in (0,\bar{\eps})$. For every $((p_1,p_{10}),(P_{k_1},\ldots,P_{k_5})) \in \left(\pa B_R(0)\right)^2 \times \wt{\mathcal{P}}^5$ the function $\sigma^{((p_1,p_{10}),(P_{k_1},\ldots, P_{k_5}),\eps,0)}$ is collision-free during its third passage inside the ball $B_R(0)$. 
\end{lemma}

\noindent We denote with  $T(\s)$ or $T_{(p_2,\ldots,p_9)}^{((p_1,p_{10}),(P_{k_1},\ldots, P_{k_5}),\eps,\nu')}$ the maximum of the time interval of \\
$\s=\sigma_{(p_2,\ldots,p_9)}^{((p_1,p_{10}),(P_{k_1},\ldots, P_{k_5}),\eps,\nu')}$. We collect the boundedness properties of outer and inner solutions, see Lemma \ref{bound per i tempi esterni} and Corollary \ref{bound tempi interni}.

\begin{lemma}\label{limitatezza funzioni incollate}
Let $\eps \in (0,\bar{\eps})$. There are $C_1, C_2, C_3 >0$ such that
\begin{gather*}
C_1 \leq T_{(p_2,\ldots,p_9)}^{((p_1,p_{10}),(P_{k_1},\ldots, P_{k_5}),\eps,\nu')} \leq C_2 \\
\|\sigma_{(p_2,\ldots,p_9)}^{((p_1,p_{10}),(P_{k_1},\ldots, P_{k_5}),\eps,\nu')}\|_{H^1([0,T(\sigma)])} \leq C_3 
\end{gather*}
for every $((p_2,\ldots,p_9),(p_1,p_{10}),(P_{k_1},\ldots,P_{k_5}),\nu') \in B \times \left(\pa B_R(0)\right)^2 \times \wt{\mathcal{P}}^5 \times (-\bar{\nu}',\bar{\nu}')$.
\end{lemma}

\noindent It is preferable to deal with functions defined in the same time interval. Therefore, for every $\sigma=\sigma_{(p_2,\ldots,p_9)}^{((p_1,p_{10}),(P_{k_1},\ldots, P_{k_5}),\eps,\nu')}$ we introduce the re-parameterization $v(t):=v_{(p_2,\ldots,p_9)}^{((p_1,p_{10}),(P_{k_1},\ldots, P_{k_5}),\eps,\nu')}(t)=\sigma_{(p_2,\ldots,p_9)}^{((p_1,p_{10}),(P_{k_1},\ldots, P_{k_5}),\eps,\nu')}(T(\sigma)t)$, for $t \in [0,1]$.

\begin{definition}
We collect the "glued function" $v$ in 
\begin{multline*}
\mathcal{GF}_\eps:= \left\{ v=v_{(p_2,\ldots,p_9)}^{((p_1,p_{10}),(P_{k_1},\ldots, P_{k_5}),\eps,\nu')} \text{for some } (p_2,\ldots,p_9) \in B, \right. \\
\left. (p_1,p_{10}) \in \left(\pa B_R(0)\right)^2, \ (P_{k_1},\ldots, P_{k_5}) \in \wt{\mathcal{P}}^5,\ |\nu'|<\bar{\nu}' \right\}.
\end{multline*}
For each $v \in \mathcal{GF}_\eps$ we term
\[
\o(v)^2:= \frac{\int_0^1 \Phi_{\nu',\eps}(v)-1}{\frac{1}{2} \int_0^1 |\dot{v}|^2}.
\]
\end{definition}
\noindent Note that, if $v(t)=\s(T(\s)t)$, then $\o(v)=1/T(\s)$. Note also that for every $\eps \in (0,\bar{\eps})$ there exists $C>0$ such that $\|v \|_{H^1} \leq C$ for every $v \in \mathcal{GF}_{\eps}$. It follows from Lemma \ref{limitatezza funzioni incollate}, taking into account the boundedness properties for the time intervals of inner and outer solutions. %It will be useful to stress the dependence of $v$ and of $\sigma$ only on $\nu'$. In such a situation, we will use the notation $v_{\nu'}$ and $\sigma_{\nu'}$ respectively. 
In order to work with sequences of functions in $\mathcal{GF}_\eps$, it is convenient to introduce some notation. Fixed $(P_{k_1},\ldots,P_{k_5}) \in \mathcal{P}^5$ and $\eps \in (0,\bar{\eps})$, assume that we have $((p_2^m,\ldots,p_9^m))_m \subset B$, $((p_1^m,p_{10}^m))_m \subset \left(\pa B_R(0)\right)^2$, $(\nu'_m) \subset (-\bar{\nu}',\bar{\nu}')$ such that
\[
(p_2^m,\ldots,p_9^m) \to (\wh{p}_2,\ldots,\wh{p}_9) \qquad (p_1^m,p_{10}^m) \to (\wh{p}_1,\wh{p}_{10}) \qquad \nu_m' \to 0.
\]
We will use the following notations
\begin{gather}
v_m:=v_{(p_2^m,\ldots,p_9^m)}^{((p_1^m,p_{10}^m), (P_{k_1},\ldots,P_{k_5}),\eps,\nu_m')} \qquad \o_m:=\o(v_m) \label{eq27} \\
\sigma_m:= \s_{(p_2^m,\ldots,p_9^m)}^{((p_1^m,p_{10}^m), (P_{k_1},\ldots,P_{k_5}),\eps,\nu_m')} \qquad T_m:=T(\s_m);
\end{gather}
Subscripts will be replaced by the accent $\wh{\cdot}$ for the function corresponding to the limit points. Recall that $\s_m$ has been obtained by the juxtaposition of 
\[
y_{P_{k_{j+1}}}(\cdot\,;p_{2j+1}^m,p_{2j+2}^m;\eps,\nu_m')=: y_{2j+1}^m \quad \text{and} \quad  y_{\text{ext}}(\cdot\,;p_{2j}^m,p_{2j+1}^m;\eps,\nu_m')=:y_{2j}^m.
\]
Each $y_j^m$ is defined over a time interval $[0,T_j^m]$. There are $0=t_1^m<t_2^m<\ldots t_9^m <t_{10}^m=T(\s_m)$ such that $\s_m(t_k^m)=p_k^m$ for every $k=1,\ldots,10$. We have $T_j^m=t_{j+1}^m-t_j^m$. For $j=0,\ldots,4$, recall that
\[
y_{P_{k_{j+1}}}(\cdot\,;p_{2j+1}^m,p_{2j+2}^m;\eps,\nu_m')=u_{P_{k_{j+1}}}\left(\frac{\cdot}{T_{2j+1}^m}\cdot\,;p_{2j+1}^m,p_{2j+2}^m;\eps,\nu_m'\right)=:u_{2j+1}^m.
\]

\begin{lemma}\label{GF_eps è debolmente chiuso}
Let $\eps \in (0,\bar{\eps})$, $(P_{k_1},\ldots,P_{k_5}) \in \mathcal{P}^5$. Assume that we have sequences $((p_2^m,\ldots,p_9^m))_m \subset B$, $((p_1^m,p_{10}^m))_m \subset \left(\pa B_R(0)\right)^2$, $(\nu'_m) \subset (-\bar{\nu}',\bar{\nu}')$ such that
\[
(p_2^m,\ldots,p_9^m) \to (\wh{p}_2,\ldots,\wh{p}_9) \qquad (p_1^m,p_{10}^m) \to (\wh{p}_1,\wh{p}_{10}) \qquad \nu_m' \to 0.
\]
Using the notations previously introduced, assume that exists $v \in H^1([0,1])$ such that $v_m \wc v$ weakly in $H^1$. Then 
\[
v=v_{(\wh{p}_2,\ldots,\wh{p}_9)}^{((\wh{p}_1,\wh{p}_{10}),(P_{k_1},\ldots, P_{k_5}),\eps,0)}.
\]
\end{lemma}

\begin{proof}
Under the convergence of the ends and of $\nu_m'$, inner and outer solutions $y_k^m$ are weakly convergent to inner and outer solutions $\wh{y}_k$ (see Propositions \ref{teorema 0.1} and the Continuity Lemma \ref{minimi interni con variazione di nu}); the thesis follows easily.
\end{proof}

\noindent To each $((p_1,p_{10}),(P_{k_1},\ldots,P_{k_5}),\nu') \in \left(\pa B_R(0)\right)^2 \times \wt{\mathcal{P}}^5 \times (-\bar{\nu}_1',\bar{\nu}_1')$ we can associate an element of $\mathcal{GF}_\eps$ in the following way: it is well defined the function $\mathfrak{F}_{((p_1,p_{10}),(P_{k_1},\ldots, P_{k_5}),\eps,\nu')}$, and we know that it has a minimum. To a minimum we associated the function $\sigma^{((p_1,p_{10}),(P_{k_1},\ldots, P_{k_5}),\eps,\nu')}$, which can be re-parametrized obtaining $v^{((p_1,p_{10}),(P_{k_1},\ldots, P_{k_5}),\eps,\nu')}$. We are ready to state the counterpart of the Continuity Lemma \ref{minimi interni con variazione di nu}.

\begin{con lemma}\label{continuity lemma 2}
Let $\eps \in (0,\bar{\eps})$, $(P_{k_1},\ldots,P_{k_5})\in \mathcal{P}^5$, $((p_1^m,p_{10}^m) ) \subset \left(\pa B_R(0)\right)^2$ and $(\nu_m') \subset (-\bar \nu',\bar \nu')$. Let $v_m=v^{((p_1^m,p_{10}^m),(P_{k_1},\ldots, P_{k_5}),\eps,\nu_m')}$ be a function of $\mathcal{GF}_\eps$ associated to a minimizer of the following variational problem:
\[
\min \left\{ \mathfrak{F}_{((p_1^m,p_{10}^m),(P_{k_1},\ldots, P_{k_5}),\eps,\nu_m')}(p_2,\ldots,p_9): (p_2,\ldots,p_9) \in B \right\}.
\]
Assume $(p_1^m,p_{10}^m) \to (\wt{p}_1,\wt{p}_{10})$, $\nu_m' \to 0$, and $v_m \wc \wt{v}$ weakly in $H^1$. Then $\wt{v}$ is the function associated to a minimizer for
\[
\min \left\{\mathfrak{F}_{((\wt{p}_1,\wt{p}_{10}),(P_{k_1},\ldots, P_{k_5}),\eps,0)}(p_2,\ldots,p_9): (p_2,\ldots,p_9) \in B \right\}.
\]
\end{con lemma}

\noindent This continuity result permits to prove the following proposition, which is the last step in the proof of Theorem \ref{esistenza di sol periodiche 3}.

\begin{proposition}\label{no coll 3+1 corpi}
Let $\a=1$ and $N = 3$. Let $\eps \in (0,\bar{\eps})$. There exists $\bar{\nu}_2'(\eps)$ such that for every $\nu'\in(-\bar{\nu}_2'(\eps),\bar{\nu}_2'(\eps))$, $n \in \N$ and $(P_{j_1},\ldots,P_{j_{4n}}) \in \mathcal{G}^n$, the function $\g^{((P_{j_1},\ldots,P_{j_{4n}}),\eps,\nu')}$ is collision-free.
\end{proposition}

\begin{proof}
Let  $(P_{j_1},\ldots,P_{j_{4n}}) \in \mathcal{G}^n$ and $\nu' \in (-\bar{\nu}',\bar{\nu}')$. Let us consider the restriction of \\
$\g=\gamma^{((P_{j_1},\ldots,P_{j_{4n}}),\eps,\nu')}$ in a time interval $[s_1,s_2]$, chosen in such a way that $\g|_{[s_1,s_2]}$ describes one passage of $\g$ inside $B_R(0)$. The goal is to show that $\g|_{[s_1,s_2]}$ is collision-free. There are
\begin{itemize}
\item $t_k \in \R$ and $p_k \in \pa B_R(0)$ such that $\g(t_k)=p_k$, for every $k=1,\ldots,10$.
\item $(P_{k_1},\ldots,P_{k_5}) \in \mathcal{P}^5$, 
\end{itemize}
 such that $\g|_{[t_1,t_{10}]}=\s^{((p_1,p_{10}),(P_{k_1},\ldots,P_{k_5}),\eps,\nu')}=\s$ and $\g|_{[s_1,s_2]}=\s|_{[t_5,t_6]}$, where $t_5$ and $t_6$ have been defined in \eqref{eq15}. This means that each passage of $\g$ inside $\pa B_R(0)$ is the third passage inside $\pa B_R(0)$ of a function $\s^{((p_1,p_{10}),(P_{k_1},\ldots,P_{k_5}),\eps,\nu')}$, for some $p_1,p_{10} \in \pa B_R(0)$ and $(P_{k_1},\ldots,P_{k_5}) \in \mathcal{P}^5$. This observation is the key point of the proof: it implies that our thesis follows if we show that there are $\bar{\nu}_2',\beta_2>0$ such that 
\beq\label{eq21}
\min_{k \in \{1,\ldots,N\}} \left(\min_{t \in \left[\frac{t_5}{T(\sigma)},\frac{t_6}{T(\sigma)}\right]} |v^{((p_1,p_{10}),(P_{k_1},\ldots, P_{k_5}),\eps,\nu')}(t)-c_{k_3}|\right) \geq \beta_2
\eeq
for every $((p_1,p_{10}),(P_{k_1},\ldots,P_{k_5}),\nu') \in \left(\pa B_R(0)\right)^2 \times \wt{\mathcal{P}}^5 \times (-\bar{\nu}_2',\bar{\nu}_2')$; this implies that \\
$v^{((p_1,p_{10}),(P_{k_1},\ldots, P_{k_5}),\eps,\nu')}$ (and hence $\s^{((p_1,p_{10}),(P_{k_1},\ldots, P_{k_5}),\eps,\nu')}$) cannot have a collision in its third passage inside $B_R(0)$, independently on $(p_1,p_{10})$ and $(P_{k_1},\ldots,P_{k_5})$. \\
Assume by contradiction that \eqref{eq21} is not true. Then there are $(\beta_m) \subset \R^+$, $(\nu'_m) \subset (-\bar{\nu}',\bar{\nu}')$, $\left((p_1^m,p_2^m)\right) \subset \left(\pa B_R(0)\right)^2$, $\left((P_{k_1},\ldots,P_{k_5})^m \right) \subset  \wt{\mathcal{P}}^5$  such that $\beta_m \to 0$, $\nu_m'\to 0$ for $m \to \infty$, and 
\[
\min_{t \in \left[\frac{t_5^m}{T(\sigma_m)},\frac{t_6^m}{T(\sigma_m)}\right]} |v^{((p_1^m,p_{10}^m),(P_{k_1},\ldots, P_{k_5})^m,\eps,\nu'_m)}(t)-c_{k_3^m}| = \beta_m \qquad \forall m.
\]
Since $\wt{\mathcal{P}}^5$ is discrete and finite, we can assume $(P_{k_1},\ldots,P_{k_5})^m=(P_{k_1},\ldots,P_{k_5})$ for every $m$. Also, since $\pa B_R(0)$ is compact, up to a subsequence $(p_1^m,p_2^m) \to (\wh{p}_1,\wh{p}_2) \in \pa B_R(0)$. We term $v_m=v^{((p_1^m,p_{10}^m),(P_{k_1},\ldots, P_{k_5})^m,\eps,\nu'_m)}$. The image of $v_m$ intersects the circle $\pa B_R(0)$ in 8 points $(p_2^m,\ldots,p_9^m) \in B$ in succession. Up to a subsequence $(p_2^m,\ldots,p_9^m) \to (\wh{p}_2,\ldots,\wh{p}_9)$.  We observed that the set $\mathcal{GF}_\eps$ is bounded in the $H^1$ norm, therefore up to a subsequence $v_m \wc \wh{v} \in H^1([0,1])$ weakly in $H^1$ (and hence uniformly). The image of $\wh{v}$ intersects the circle in the $8$ points $(\wh{p}_2,\ldots,\wh{p}_9)$ in succession. To be precise
\[
\wh{v}= v_{(\wh{p}_2,\ldots,\wh{p}_9)}^{((\wh{p}_1,\wh{p}_{10}),(P_{k_1},\ldots, P_{k_5}),\eps,0)} \in \mathcal{GF}_\eps,
\]
see Lemma \ref{GF_eps è debolmente chiuso}. By the Continuity Lemma \ref{continuity lemma 2}, the point $(\wh{p}_2,\ldots,\wh{p}_9)$ minimizes $\mathfrak{F}_{((\wh{p}_1,\wh{p}_{10}),(P_{k_1},\ldots, P_{k_5}),\eps,0)}$ in $B$. But the uniform convergence implies that $\wh{v}$ has a collision in its third passage inside $B_R(0)$, and this is in contradiction with Lemma \ref{no coll 3-centri}.
\end{proof}

\paragraph{3) $\a\in (1,2)$.} This is the easiest case, since for every $\eps \in (0,\bar{\eps})$, $n\in\mathbb{N}$, $(P_{j_1},\ldots,P_{j_n}) \in\mathcal{P}^n$ the weak solution $\g^{((P_{j_1},\ldots,P_{j_n}),\eps,0)}$ is collision-free (Theorem 5.3 of \cite{SoTe}). Thus, we can simply follows the sketch already developed for point 1) with minor changes.

\begin{proposition}
Let $\a \in (1,2)$. Let $\eps \in (0,\bar{\eps})$. There exists $\bar{\nu}_3'(\eps)$ such that for every $\nu'\in(-\bar{\nu}_3'(\eps),\bar{\nu}_3'(\eps))$, $n \in \N$ and $(P_{j_1},\ldots,P_{j_n}) \in \mathcal{P}^n$, the function $\g^{((P_{j_1},\ldots,P_{j_n}),\eps,\nu')}$ is collision-free.
\end{proposition}

\section{Proofs of the continuity lemmas}\label{dim con lemma}

\subsection{Proof of Continuity Lemma \ref{minimi interni con variazione di nu}}

Let $u_0$ be a minimizer of $M_0$ in  $K_{P_j}^{\wt{p}_1 \wt{p}_2}([0,1])$. We aim at proving that $M_0(\wt{u})=M_0(u_0)$. We will briefly write $L_m$ for $L_{\nu_m'}$ and $M_m$ for $M_{\nu_m'}$.\\
The following statement is a continuity property for the functionals $\{ M_m\}$ in the set of the minimizers $\{ u_m\}$. 

\begin{lemma}\label{continuità uniforme sui minimi}
The family $\{M_m\}_m$ tends to $M_0$ as $m \to \infty$, uniformly in the set $\{u_m:m \in \N\}$. This means that for every $\l>0$ exists $m_1 \in \mathbb{N}$ such that
\[
m>m_1  \Rightarrow  |M_m(u_{\bar{m}})-M_0(u_{\bar{m}})| \leq \l \quad \forall \bar{m} \in \N.
\] 
\end{lemma}
\begin{proof}
Let $\bar{m} \in \N$. For every $m$ we have 
\begin{multline*}
|M_m(u_{\bar{m}})-M_0(u_{\bar{m}})| \leq |\nu_m'| \int_0^1 |u_{\bar{m}}||\dot{u}_{\bar{m}}|\\
+\sqrt{2} \left(\int_0^1 |\dot{u}_{\bar{m}}|\right)^{\frac{1}{2}} \left| \left(\int_0^1 V_{\eps}(u_{\bar{m}})-1+\frac{(\nu_m')^2}{2} |u_{\bar{m}}|^2 \right)^{\frac{1}{2}}- \left( \int_0^1 V_{\eps}(u_{\bar{m}})-1 \right)^{\frac{1}{2}} \right|
\end{multline*}
Let $\f_{\bar{m}}(\nu):= \left(\int_0^1 V_\eps(u_{\bar{m}})-1+\frac{(\nu^2)}{2} |u_{\bar{m}}|^2 \right)^{1/2}$. It results
\[
|\f_{\bar{m}}(\nu_m')-\f_{\bar{m}}(0)| \leq \frac{1}{2} \left(\int_0^1 V_\eps(u_{\bar{m}})-1\right)^{-\frac{1}{2}} \int_0^1 |u_{\bar{m}}|^2 (\nu_m')^2 \leq \frac{R^2}{2 \sqrt{M_1}} (\nu_m')^2 ,
\] 
so that
\[
|M_m(u_{\bar{m}})-M_0(u_{\bar{m}})| \leq R \|\dot{u}_{\bar{m}}\|_{L^2} |\nu_m'|+ \frac{R^2}{\sqrt{2  M_1}}\|\dot{u}_{\bar{m}}\|_{L^2}(\nu_m')^2 \leq C(|\nu_m'| + (\nu_m')^2),
\]
where $C$ is a constant independent on $\bar{m}$ (see Proposition \ref{minimi limitati}). 
\end{proof}

We want to compare $M_m(u_m)$ with $M_m(u_0)$. Because of the minimality property of $u_m$ it seems reasonable to think that $M_m(u_m) \leq M_m(u_0)$. This is not immediate, and not necessarily true, since $u_m$ is a minimizer of $M_m$ for the fixed ends problem $\min \{M_m(u): u \in K_{P_j}^{p_1^m p_2^m}([0,1])\}$, while $u_0$ connects $\wt{p}_1$ and $\wt{p}_2$. However, the fact that $p_1^m \to \wt{p}_1$ and $p_2^m \to \wt{p}_2$ suggests that maybe we can prove something similar (which in fact will be equation \eqref{eq17}). For every $p_* ,p_{**} \in \pa B_R(0)$ we consider again the function $\zeta_R(\cdot\,;p_*,p_{**})$ which parametrizes the shorter arc of $\pa B_R(0)$ connecting $p_*$ and $p_{**}$ in time $1$ with constant angular velocity. It is easy to check that
\[
\forall \l>0 \ \exists \rho>0: \ |p_*-p_{**}|<\rho \Rightarrow M_0(\zeta_R(\cdot\,;p_*,p_{**}))<\l,
\]
so that 
\beq\label{eq16}
\forall \l >0 \ \exists m_2 \in \N: \ m>m_2 \Rightarrow  \begin{cases} M_0(\zeta_R(t;p_1^m,\wt{p}_1)) < \l \\
M_0(\zeta_R(t;\wt{p}_2,p_2^m)) < \l. \end{cases} 
\eeq
Furthermore, the following continuity property holds true.

\begin{lemma}\label{continuità uniforme sugli archetti}
The family $\{M_m\}_m$ tends to $M_0$ as $m \to \infty$, uniformly in the set $\{\zeta_R(\cdot\,;p_*,p_{**}):\ p_{*},p_{**} \in \pa B_R(0)\}$. This means that for every $\l>0$ exists $m_3 \in \mathbb{N}$ such that
\[
m>m_3  \Rightarrow  |M_m(\zeta_R(\cdot\,;p_*,p_{**}))-M_0(\zeta_R(\cdot\,;p_*,p_{**}))| \leq \l \quad \forall p_*,p_{**}\in \pa B_R(0).
\] 
\end{lemma}

\begin{proof}
We can adapt the proof of Lemma \ref{continuità uniforme sui minimi} with minor changes.
\end{proof}

\begin{proof}[Conclusion of the proof of the Continuity Lemma \ref{minimi interni con variazione di nu}]
Because of the minimality of $u_0$ and the weak lower semi-continuity of $M_0$ it results
\beq\label{eq18}
M_0(u_0) \leq M_0(\wt{u}) \leq \liminf_{m \to \infty} M_0(u_m).
\eeq
For every $m \in \N \cup\{0\}$ we have
\[
\frac{\o_m^2}{2} |\dot{u}_m|^2-\Phi_{\nu_m',\eps}(u_m)=-1 \quad \text{a.e. in $[0,1]$} \ \Rightarrow \ \sqrt{2} L_m(u_m)=M_m(u_m),
\]
where $\o_m=\o_{P_j}(p_1^m,p_2^m;\eps,\nu_m')$. The variational characterization of $u_m$ implies that
\begin{multline}\label{eq17}
M_m(u_m) = \sqrt{2} L_m(u_m) \leq   \sqrt{2} L_m(\zeta_R(\cdot\,;p_1^m,\wt{p}_1))+ \sqrt{2} L_m(u_0) + \sqrt{2}  L_m(\zeta_R(\cdot\,;\wt{p}_2,p_2^m)) \\
 \leq M_m(\zeta_R(\cdot\,;p_1^m,\wt{p}_1))+  M_m(u_0) + M_m(\zeta_R(\cdot\,;\wt{p}_2,p_2^m)).
\end{multline}
We passed to the functional $L_m$ in order to exploit its additivity property, which does not hold for $M_m$. Lemmas \ref{continuità uniforme sui minimi}, \ref{continuità uniforme sugli archetti} and equation \eqref{eq16} imply that for every $\l >0$  if $m > \max\{m_1,m_2,m_3\}$ then
\[
\begin{cases}
M_m(u_m) > M_0(u_m)-\l \\
M_m(\zeta_R(\cdot\,;p_1^m,\wt{p}_1))<M_0(\zeta_R(\cdot\,;p_1^m,\wt{p}_1))+\l< 2\l \\
M_m(\zeta_R(\cdot\,;\wt{p}_2,p_2^m))<M_0(\zeta_R(\cdot\,;\wt{p}_2,p_2^m))+\l< 2\l \\
M_m(u_0)<M_0(u_0)+\l.
\end{cases}
\]
Hence, from equation \eqref{eq17}, for every $\l>0$ if $m > \max\{m_1,m_2,m_3\}$ then 
\[
M_0(u_m)-\l \leq M_0(u_0)+5\l \Rightarrow \limsup_{m \to \infty} M_0(u_m) \leq M_0(u_0).
\]
This, together with \eqref{eq18}, says that the sequence $(M_0(u_m))_m$ has a limit and $M_0(u_0)=M_0(\wt{u})=\lim_m M_0(u_m)$; in particular $\wt{u}$ is a minimizer of $M_0$ in $K_{P_j}^{\wt{p}_1 \wt{p}_2}([0,1])$.
\end{proof}

\subsection{Proof of the Continuity Lemma \ref{continuity lemma 2}}

Let $\sigma_0=\sigma^{((\wt{p}_1,\wt{p}_{10}),(P_{k_1},\ldots, P_{k_5}),\eps,0)}=\sigma_{(\wh{p}_2,\ldots,\wh{p}_9)}^{((\wt{p}_1,\wt{p}_{10}),(P_{k_1},\ldots, P_{k_5}),\eps,0)}$, where $(\wh{p}_2,\ldots,\wh{p}_9)$ is a minimizer of $\mathfrak{F}_{((\wt{p}_1,\wt{p}_{10}),(P_{k_1},\ldots, P_{k_5}),\eps,0)}$, and let $v_0(t)=\s_0(T(\s_0 t))$. We aim at proving that $M_0(\wt{v})=M_0(v_0)$. We need two intermediate results. The first one is a generalization of Lemma \ref{continuità uniforme sui minimi} for the glued functions.

\begin{lemma}\label{continuita' uniforme sulle funzioni incollate}
Let $(v_m) \subset \mathcal{GF}_\eps$, where each $v_m$ is a glued function defined by \eqref{eq27}. The family $\{M_m\}_m$ tends to $M_0$ for $m \to \infty$, uniformly in $\{v_m\}_m$. This means that for every $\l>0$ exists $m_1 \in \N$ such that 
\[
m>m_1 \Rightarrow |M_m(v_{\bar{m}})-M_0(v_{\bar{m}})|<\l \qquad \forall \bar{m}.
\]
\end{lemma}
\begin{proof}
We can adapt the proof of Lemma \ref{continuita' uniforme sulle funzioni incollate}; the only difference is that we used the uniform bounds 
\[
\|u\|_{L^2} \leq R \qquad \|\dot{u}\|_{L^2}\leq C \qquad \int_0^1 V_{\eps}(u)-1 \geq M_1 \qquad \forall u \in \mathcal{IM}_\eps.
\]
Now we are considering glued functions, so we need similar properties for the function of $\mathcal{GF}_\eps$. We have already noticed that there is $C>0$ such that $\|\dot{v}_{\bar{m}}\|_{H^1} \leq C$ for every $\bar{m}$; furthermore, 
\[
\int_0^1 V_{\eps}( v_{\bar{m}} )-1 \geq \frac{1}{T(\sigma_{\bar{m}})}\sum_{j=1}^4 \int_0^{T_{2j+1}}\left(V_\eps(y_{2j+1})-1\right)   \geq \frac{4M_1}{C}. \qedhere
\]
\end{proof}

\begin{lemma}\label{continuità lunghezza esterna}
Let $p_{2j},p_{2j+1} \in \pa B_R(0)$ be such that $|p_{2j}-p_{2j+1}| \leq \d$,  let $(\nu_m') \subset (-\bar{\nu}',\bar{\nu}')$ be such that $\nu_m' \to 0$ as $m \to \infty$. \\
For every $\l>0$ there exists  $m_4=m_4(p_{2j},p_{2j+1}) \in \N$ such that
\[
|L_{\nu_{\bar{m}}'}(y_{\text{ext}}(\cdot\,;p_{2j},p_{2j+1};\eps,\nu_m'))- L_{\nu_{\bar{m}}'}(y_{\text{ext}}(\cdot\,;p_{2j},p_{2j+1};\eps,0))| < \l 
\]
for every $\bar{m} \in \N$.
\end{lemma}
\begin{proof}
We will write $y_m$ instead of $y_{\text{ext}}(\cdot\,;p_{2j},p_{2j+1};\eps,\nu_m')$ to ease the notation. Let $T_m$ be such that $y_m(T_m)=p_{2j+1}$.
\begin{multline}\label{eq28*}
|L_{\bar{m}}(y_m)-L_{\bar{m}}(y_0)| \leq \left| \int_0^{T_m} \sqrt{\Phi_{\nu_{\bar{m}}',\eps}(y_m(t))-1} |\dot{y}_m(t)|\,dt  - \int_0^{T_0} \sqrt{\Phi_{\nu_{\bar{m}}',\eps}(y_0(t))-1} |\dot{y}_0(t)|\,dt \right| \\
+ \left| \int_0^{T_m} \langle i y_m(t),\dot{y}_m(t) \rangle\,dt - \int_0^{T_0} \langle i y_0(t),\dot{y}_0(t)\rangle \,dt \right|.
\end{multline}
We have already observed (Remark \ref{remark 7}) that $\int_0^1 \langle i u, \dot{u} \rangle$ is continuous in the weak topology of $H^1$. We know that $y_m \to y_0$ $\mathcal{C}^1$-uniformly; it is not difficult to check that consequently 
\beq\label{eq28}
y_m(T_m t) \to y_0(T_0 t) \quad \text{$\mathcal{C}^1$-uniformly in $[0,1]$},
\eeq
so that the second term in the right hand side of \eqref{eq28*} tends to $0$ as $m \to \infty$ (independently on $\bar{m}$). As far as the first term on the right hand side of \eqref{eq28*} is concerned, it results
\begin{multline}\label{eq30}
\left| \int_0^{T_m}  \sqrt{\Phi_{\nu_{\bar{m}}',\eps}(y_m(t))-1} |\dot{y}_m(t)|\,dt  - \int_0^{T_0} \sqrt{\Phi_{\nu_{\bar{m}}',\eps}(y_0(t))-1} |\dot{y}_0(t)|\,dt \right| \\
 \leq \int_0^1 \left| \sqrt{\Phi_{\nu_{\bar{m}}',\eps}(y_m(T_m t))-1} - \sqrt{\Phi_{\nu_{\bar{m}}',\eps}(y_0(T_0 t))-1} \right| |\dot{y}_m(T_m t)|\,dt  \\
+ \int_0^1 \sqrt{\Phi_{\nu_{\bar{m}}',\eps}(y_0(T_0 t))-1} \left| |\dot{y}_0(T_0 t)|-|\dot{y}_m(T_m t)| \right| \,dt.
\end{multline}
The function $\sqrt{\cdot}$ is $1/2$-H\"{o}lder continuous, so that for every $\bar{m}$
\begin{multline}\label{eq31}
\int_0^1 \left| \sqrt{\Phi_{\nu_{\bar{m}}',\eps}(y_m(T_m t))-1} - \sqrt{\Phi_{\nu_{\bar{m}}',\eps}(y_0(T_0 t))-1}\right| |\dot{y}_m(T_m t)|\,dt  \\
\leq \left(\int_0^1 \left|\sqrt{\Phi_{\nu_{\bar{m}}',\eps}(y_m(T_m t))-1} - \sqrt{\Phi_{\nu_{\bar{m}}',\eps}(y_0(T_0 t))-1}\right|^2\,dt\right)^{\frac{1}{2}} \| \dot{y}_m(T_m \cdot)\|_{L^2} \\
\leq C \left( \int_0^1 |\Phi_{\nu_{\bar{m}}',\eps}(y_m(T_m t))- \Phi_{\nu_{\bar{m}}',\eps}(y_0(T_0 t))|\,dt\right)^{\frac{1}{2}};
\end{multline}
In the last inequality, we took advantage of the uniform bound for the $L^2$ norm of outer solutions. Both $y_m$ and $y_0$ are outer solutions, therefore we can exploit the fact that $V_{\eps}$ is $\mathcal{C}^\infty$ with bounded derivatives outside $\pa B_R(0)$; using also \eqref{eq28} and the first estimate \eqref{limitazione soluzioni esterne}, we obtain
\beq\label{eq31*}
\sup_{t \in [0,1]} | \Phi_{\nu_{\bar{m}}',\eps}(y_m(T_m t))- \Phi_{\nu_{\bar{m}}',\eps}(y_0(T_0 t)) | \leq C(1 +  |\nu_{\bar{m}}'|^2) \sup_{t \in [0,1]} |y_m(T_m t)-y_0(T_0 t)| \to 0
\eeq
as $m \to \infty$, independently on $\bar{m}$ (recall that $|\nu_{\bar{m}}'| \leq \bar{\nu}'$). Furthermore, using again \eqref{eq28} it is easy to check
\beq\label{eq32}
\int_0^1 \sqrt{\Phi_{\nu_{\bar{m}}',\eps}(y_0(T_0 t))-1} \left| |\dot{y}_0(T_0 t)|-|\dot{y}_m(T_m t)| \right| \,dt  \to 0,
\eeq 
as $m \to \infty$, independently on $\bar{m}$. Collecting \eqref{eq31}, \eqref{eq31*}, \eqref{eq32} and comparing with \eqref{eq30} we deduce that also the first term on the right hand side of \eqref{eq28*} tends to $0$, uniformly in $\bar{m}$.
\end{proof}

\begin{proof}[Conclusion of the proof of the Continuity Lemma \ref{continuity lemma 2}]
The conservation of the Jacobi constant holds true both for $v_0$ and $\wt{v}$ (recall that $\wt{v}\in \mathcal{GF}_\eps$, as showed in Lemma \ref{GF_eps è debolmente chiuso}); using this property, the minimality of $\sigma_0$ and the weak lower semi-continuity of $M_0$,  we have 
\beq\label{eq22}
M_0(v_0)=L_0(v_0) \leq L_0(\wt{v}) = M_0(\wt{v}) \leq \liminf_{m \to \infty} M_0(v_m) .
\eeq
We pose $\wh{p}_1:=\wt{p}_1$ and $\wh{p}_{10}:=\wt{p}_{10}$. The minimality of $(p_2^m,\ldots,p_9^m)$ for \\ $\mathfrak{F}_{((p_1^m,p_{10}^m),(P_{k_1},\ldots,P_{k_5}),\eps,\nu_m')}$ implies that
\begin{equation}\label{eq24}
 \begin{split}
M_m(\sigma_m)  &=\sqrt{2}L_m(\s_m) \leq \sqrt{2} L_m(\s_{(\wh{p}_2,\ldots,\wh{p}_9)}^{((p_1^m,p_{10}^m),(P_{k_1},\ldots,P_{k_5}),\eps,\nu_m')} ) \\
 &\leq \sqrt{2} \left(L_m(\s_{(\wh{p}_2,\ldots,\wh{p}_9)}^{((\wh{p}_1,\wh{p}_{10}),(P_{k_1},\ldots,P_{k_5}),\eps,\nu_m')} )+ L_m(\zeta_R(\cdot\,;p_1^m,\wh{p}_1)) + L_m(\zeta_R(\cdot\,;\wh{p}_{10},p_{10}^m)) \right) \\
&=  \sqrt{2} \left( \sum_{j=0}^4 L_m(y_{P_{k_{j+1}}}(\cdot\,;\wh{p}_{2j+1},\wh{p}_{2j+2};\eps,\nu_m'))+ \sum_{j=1}^4 L_m(y_{\text{ext}}(\cdot\,;\wh{p}_{2j},\wh{p}_{2j+1};\eps,\nu_m')) \right. \\
& \quad +  L_m(\zeta_R(\cdot\,;p_1^m,\wh{p}_1)) + L_m(\zeta_R(\cdot\,;\wh{p}_{10},p_{10}^m)) \Bigg)  
 \end{split}
\end{equation}
The variational characterization of $y_{P_{k_{j+1}}}(\cdot\,;\wh{p}_{2j},\wh{p}_{2j+1};\eps,\nu_m')$ implies that 
\[
L_m(y_{P_{k_{j+1}}}(\cdot\,;\wh{p}_{2j+1},\wh{p}_{2j+2};\eps,\nu_m')) \leq L_m(y_{P_{k_{j+1}}}(\cdot\,;\wh{p}_{2j+1},\wh{p}_{2j+2};\eps,0)).
\]
Also, let us collect the uniform estimates of equation \eqref{eq16}, Lemmas \ref{continuità uniforme sugli archetti}, \ref{continuita' uniforme sulle funzioni incollate} and \ref{continuità lunghezza esterna}: for every $\l>0$ exists $m_5:= \max\{m_1,\ldots,\max\{m_4(\wh{p}_{2j},\wh{p}_{2j+1}):j =1,\ldots,4\}\}$ such that
\[
\begin{cases}
M_m(v_m) > M_0(v_m)-\l \\
\sqrt{2} L_m(\zeta_R(\cdot\,;p_1^m,\wh{p}_1)) \leq M_m(\zeta_R(\cdot\,;p_1^m,\wh{p}_1)) < 2\l\\
\sqrt{2} L_m(\zeta_R(\cdot\,;\wh{p}_{10},p_{10}^m)) \leq M_m(\zeta_R(\cdot\,;\wh{p}_{10},p_{10}^m)) < 2\l \\
L_m(y_{\text{ext}}(\cdot\,;\wh{p}_{2j},\wh{p}_{2j+1};\eps,\nu_m')) < L_m(y_{\text{ext}}(\cdot\,;\wh{p}_{2j},\wh{p}_{2j+1};\eps,0))+\l \\
M_m(v_0) < M_0(v_0) + \l
\end{cases}
\]
for every $m > m_5$. Therefore, for every $\l>0$ the chain of inequalities \eqref{eq24} gives
\begin{multline*}
M_0(\s_m)-\l \leq \sqrt{2} \left(\sum_{j=0}^4 L_m(y_{P_{k_{j+1}}}(\cdot\,;\wh{p}_{2j+1},\wh{p}_{2j+2};\eps,0)) + \sum_{j=1}^4 L_m(y_{\text{ext}}(\cdot\,;\wh{p}_{2j},\wh{p}_{2j+1};\eps,0))\right) \\
+ (1+\sqrt{2})4\l = \sqrt{2} L_m(\sigma_0) + C \l \leq M_m(\sigma_0) +C \l
\end{multline*}
if $m > m_5$. With a change of variable, we can see that the previous inequality is equivalent to
\[
M_0(v_m)-\l \leq M_m(v_0)+ C \l \Rightarrow  M_0(v_m) \leq M_0(v_0)+(C+1) \l 
\]
if $m > m_5$; since $\l$ has been arbitrarily chosen, it results $\limsup_m M_0(v_m) \leq M_0(v_0)$; comparing with \eqref{eq22} we deduce that $M_0(v_0)=M_0(\wt{v})$, and the proof is complete.
\end{proof}

\noindent \textbf{Acknowledgments:} the author is indebted with Prof. Susanna Terracini for many valuable discussions related to this problem, and with an anonymous referee for his precious suggestions. This research was partially supported by PRIN 2009 grant "Critical Point Theory and Perturbative Methods for Nonlinear Differential Equations".

\end{document}